\documentclass[notitlepage,11pt,reqno]{amsart}
\usepackage[hmargin={1.0in, 1.0in}, vmargin={1.0in, 1.0in}]{geometry}
\usepackage{amssymb,amsmath,amsthm}
\usepackage[mathscr]{eucal}
\usepackage[breaklinks=true, linktocpage]{hyperref}
\usepackage{tikz-cd}
\DeclareMathOperator{\mspec}{\ensuremath{mSpec}}
\DeclareMathOperator{\hdim}{\ensuremath{hdim}}
\DeclareMathOperator{\jac}{\ensuremath{J}}
\DeclareMathOperator{\len}{\ensuremath{len}}

\theoremstyle{plain}
\newtheorem{theorem}[equation]{Theorem}
\newtheorem{lemma}[equation]{Lemma}
\newtheorem{proposition}[equation]{Proposition}
\newtheorem{corollary}[equation]{Corollary}

\theoremstyle{definition}
\newtheorem{definition}[equation]{Definition}

\newtheorem{remark}[equation]{Remark}

\newtheorem{question}[equation]{Question}

\numberwithin{equation}{section}

\makeatletter
\@namedef{subjclassname@2010}{%
  \textup{2010} Mathematics Subject Classification}
\makeatother

\begin{document}
\title{The Zariski covering number for vector spaces and modules}

\author{Soham Ghosh}
\address{Indian Statistical Institute, Bangalore 560 059, India;  Bachelor of Mathematics Student}
\email{Primary:\tt sohamghosh132001@gmail.com\\
        Secondary:\tt bmat1937@isibang.ac.in }


\date{\today}

\begin{abstract}
Given a module $M$ over a commutative unital ring $R$, let $\sigma(M, R)$ denote
the covering number, i.e.\ the smallest (cardinal) number of proper
submodules whose union covers $M$; this includes the covering
numbers of Abelian groups, which are extensively studied in the
literature. Recently, Khare--Tikaradze [\textit{Comm.\ Algebra}, in
press] showed in several cases that
$\sigma(M,R)=\min_{\mathfrak{m} \in S_M} |R / \mathfrak{m}| + 1$, where
$S_M$ is the set of maximal ideals $\mathfrak{m}$ with $\dim_{R /
\mathfrak{m}}(M / \mathfrak{m}M) \geq 2$. Our first main result extends this equality to all $R$-modules with small
Jacobson radical and finite dual Goldie dimension.

We next introduce and study a topological counterpart for finitely
generated $R$-modules $M$ over rings $R$, whose `some' residue fields are infinite, which we call the Zariski covering number
$\sigma_\tau(M,R)$. To do so, we first define the ``induced Zariski
topology'' $\tau$ on $M$, and now define $\sigma_\tau(M,R)$ to be the
smallest (cardinal) number of proper $\tau$-closed subsets of $M$ whose
union covers $M$. We then show our next main result: $\sigma_\tau(M,R) = \min_{\mathfrak{m}
\in S_M} |R / \mathfrak{m}| + 1$, for all finitely generated $R$-modules
$M$ for which (a)~the dual Goldie dimension is finite, and (b)~$\mathfrak{m}\notin S_M$ whenever $R/\mathfrak{m}$ is finite. As a corollary, this
alternately recovers the aforementioned formula for the covering number
$\sigma(M,R)$ of the aforementioned finitely generated modules.

Finally, we discuss the notion of $\kappa$-Baire spaces, and show that
the inequalities $\sigma_\tau(M,R) \leq \sigma(M,R) \leq \kappa_M :=
\min_{\mathfrak{m} \in S_M} |R / \mathfrak{m}| + 1$ again become equalities
when the image of $M$ under the continuous map $q : M \to
\prod_{\mathfrak{m} \in {\rm mSpec}(R)} M / \mathfrak{m}M$ (with
appropriate Zariski-type topologies) is a $\kappa_M$-Baire subspace of
the product space.
\end{abstract}
\maketitle


\section{Global notations}\label{S0}

We list a few of the important notations that shall be widely used throughout the paper.
\begin{enumerate}
    \item $\mathbf{K}$ will always denote a field, and $R$ will always denote a unital commutative ring.
    \item For an $R$-module $M$, $\sigma(M, R)$ denotes the \textit{covering number} of $M$, i.e., the minimum (cardinal) number of proper $R$-submodules whose union is $M$. A distinguished special case studied below is $\sigma(V, \mathbf{K})$, for a $\mathbf{K}$-vector space $V$.  
    \item For a given finitely generated $R$-module $M$ over a ring $R$, equipped with the \textit{induced Zariski topology} (defined in Section~\ref{S5}),  $\sigma_{\tau}(M, R)$ denotes the minimum (cardinal) number of proper closed subsets, whose union covers the whole space $M$. The analogous quantity for a $\mathbf{K}$-vector space equipped with the induced Zariski topology will be denoted by $\sigma_{\tau}(V, \mathbf{K})$.
    \item Given a ring $R$ and an ideal $I$, $m(I)$ denotes the set of ideals of $R$ containing $I$.
    \item $\mspec(R)$ denotes the set of maximal ideals of a given (unital commutative) ring $R$.
    \item For an $R$-module $M$, $S_{M}$ denotes the set of $\mathfrak{m}\in\mspec(R)$ for which $\dim_{R/\mathfrak{m}}(M/\mathfrak{m}M)\geq 2$.
\end{enumerate}

\section{Introduction}\label{S1}

A well-known and well-studied question in group theory is ``to find the minimum cardinal number $\sigma(G)$ of proper subgroups of a given non-cyclic group $G$ whose union covers $G$." This problem has a vast literature, much of which has been generously listed in the references of \cite{khare2020} (to which we direct any interested reader). One might ask  an analogous question for vector spaces, and in general for modules. The case of vector spaces is classical and well-known (one possible reference is \cite{KHARE20091681}). In this paper, we shall be concerning ourselves with the covering problem for finitely generated modules, and its relation to the version for vector spaces. A few instances where this problem or closely related variants have been addressed are \cite{QB}, \cite{Got}, \cite{Zandt} and most recently \cite{khare2020}. 

In \cite{khare2020}, the authors studied the minimum (cardinal) number of proper submodules required to cover a module $M$ over a unital commutative ring $R$, which we will denote in this paper by $\sigma(M, R)$. The authors of \cite{khare2020} proved the equality $\sigma(M, R)=\min_{\mathfrak{m}\in S_M}|R/\mathfrak{m}|+1$, for several classes of rings $R$ and $R$-modules, and also showed that $S_M=\emptyset$ implies cyclicity in the cases considered. Here $S_M$ is the set of maximal ideals $\mathfrak{m}$ of $R$ for which the quotient module $M/\mathfrak{m}M$ has dimension at least $2$ as a vector space over the residue field $R/\mathfrak{m}$. The authors then asked:

\begin{question}\label{Q1}
In what generality for the ring $R$, is it true that for all finitely generated non-cyclic $R$-modules $M$, we have $\sigma(M, R)=\min_{\mathfrak{m} \in S_M}|R / \mathfrak{m}|+1$?
\end{question} 

In our paper, we study a slightly different variant of this question. Our results come in two flavors -- algebraic and geometric -- both of which necessitate working with different tools and techniques than the papers in the group theory literature. To discuss our results which are of an algebraic flavor, we begin by addressing the covering problem for vector spaces over base fields $\mathbf{K}$ of arbitrary cardinalities (both finite and infinite). We provide an alternate elementary proof (using induction on dimension) for previously known results, by observing an injection from the set of hyperplanes of $\mathbf{K}^2$ to the set of hyperplanes of $\mathbf{K}^n$. The upshot of this method is that it can be easily adapted to the case of finite length semi-simple modules, which will provide a starting point for the analysis of significantly larger classes of modules. For such modules, we will consider the following modified version of Question~\ref{Q1}:

\begin{question}\label{Q2}
For which classes of finitely generated modules $M$ over any commutative ring $R$ with unity, do we have $\sigma(M, R)=\min_{\mathfrak{m} \in S_M}|R / \mathfrak{m}|+1$?
\end{question} 

We first address this problem using algebraic techniques and make use of the notion of dual Goldie dimension, introduced and studied in \cite{dgdi} and \cite{dgdii}. In particular, we show that the desired equality holds for finitely generated modules of finite dual Goldie dimension over arbitrary commutative rings with unity. One can then see that many classes of modules studied in \cite{khare2020} are special cases.

We now come to the main contribution of this paper: to discuss the covering problem from a geometric/topological viewpoint. Noting that proper subspaces of finite-dimensional vector spaces $V$ are Zariski closed, we consider the broader question of covering $V$ by the smallest number of (arbitrary) Zariski closed subsets. In the final section below, we then extend our considerations to Zariski closed subsets of finitely generated $R$-modules -- this continues to diverge from the group-theoretic aspects of covering numbers. In these endeavors, we were also motivated by \cite{KC}, where the authors introduced and investigated the notion of \textit{covering type} of a topological space $X$ from the perspective of algebraic topology. The authors showed that for finite $CW$ complexes, the notion of covering type is equivalent to that of closed covers by subcomplexes. With this equivalence in mind, we pursue a topological analysis of the covering problem for finitely generated modules by drawing a parallel between closed subcomplexes and closed submodules under our topology.

We now outline our `topological/geometric' contributions. As mentioned above, we begin by covering vector spaces by Zariski closed subsets. Utilizing an isomorphism between an $n$-dimensional vector space $V$ over an infinite base field $\mathbf{K}$ and $\mathbf{K}^n$, one can topologize $V$ (canonically) by considering $\mathbf{K}^n$ as the affine $n$-space with Zariski topology, such that the isomorphism map is a homeomorphism. We term this topology on $V$ as the \textit{induced Zariski topology}. This enables one to relate the problem of covering the topological space $V$ by closed subsets to that of covering the vector space $V$ by vector subspaces, since these are closed under the induced Zariski topology.

Finally, we equip $M$ with a compatible topology derived from maps from $M$ to $R/\mathfrak{m}$-modules $M/\mathfrak{m}M$, which we view as an $R/\mathfrak{m}$-vector space with the induced Zariski topology, when certain residue fields of $R$ are infinite. As a result, we obtain a connection between the covering problem and the notion of $\kappa$-Baire spaces for an appropriate cardinal number $\kappa$, which are a generalization of ordinary Baire spaces. Consequently, we provide a sufficient topological criterion for a class of modules to satisfy the equality in Question~\ref{Q2}.

\section{Overview of main results}\label{S2}

In this section, we present the main results proved in this paper. We begin with the covering problem for vector spaces, which we approach from a topological viewpoint in Section~\ref{S3}. First, we provide an alternate proof of the fact $\sigma (V, \mathbf{K})=|\mathbf{K}|+1$, for any finite-dimensional vector space $V$ of dimension at least $2$ over a base field $\mathbf{K}$ of arbitrary cardinality (finite or infinite). Our first main result in this work involves topologizing $V$ with the induced Zariski topology, in the case of an infinite base field. We define $\sigma_{\tau}(V, \mathbf{K})$ to be the minimum (cardinal) number of proper closed subsets required to cover the space $V$, where $V$ is equipped with the Zariski topology.

\begin{theorem}\label{introT1}
    $\sigma_{\tau}(V, \mathbf{K})=\sigma (V, \mathbf{K})=|\mathbf{K}|+1$ for any finite dimensional (at least $2$) vector space $V$ over an infinite field $\mathbf{K}$, with the induced Zariski topology.
\end{theorem}

\noindent Note that $\sigma_\tau(V, \mathbf{K})\leq \sigma(V, \mathbf{K})\leq |\mathbf{K}|+1$ is easily shown; the nontrivial assertion here is that $\sigma_{\tau}(V, \mathbf{K})\geq |\mathbf{K}|+1$.

One can additionally interpret Theorem~\ref{introT1} as a strengthening of the irreducibility of the Zariski topology on an affine space over an infinite field, which is equivalent to the inequality that $\sigma_{\tau}(V, \mathbf{K})\geq \aleph_0$, where $\aleph_0=|\mathbb{N}|$ is the smallest infinite cardinal. Alternatively, recalling the density of open sets in the Zariski topology, and switching from closed sets to open sets, the above result says that the Zariski topology on an affine $n$-space over an infinite field of cardinality $\kappa$ is a $\kappa$-maximal Baire space (see Definition~\ref{introD1}). This second interpretation shall be useful.

Our next main result is in Section~\ref{S4}, and provides a class of modules that satisfy the equality in Question~\ref{Q2}. For this, we first recall that a submodule $N$ of a given $R$-module $M$ is said to be \textit{small}, if for any other submodule $K$ of $M$, if $N+K=M$ then $K=M$.

\begin{theorem}\label{introT2}
    	Let $M$ be an $R$-module with small Jacobson radical (special case: $M$ is finitely generated) and finite dual Goldie dimension over a ring $R$. Let $S_M$ denote the set of maximal ideals $\mathfrak{m}$ of $R$ such that $\dim_{R/\mathfrak{m}}(M/\mathfrak{m}M)\geq 2$. If $S_M=\emptyset$ and $M$ is finitely generated, then $M$ is cyclic, and hence cannot be covered by proper submodules. Else if $S_M\neq\emptyset$, then $\sigma(M, R)=\min_{\mathfrak{m}\in S_M}|R/\mathfrak{m}|+1$.
\end{theorem}

The proof will be developed by starting with the case of finitely generated semi-simple modules. As corollaries of this theorem, we obtain previously known classes of modules for which the equality $\sigma(M, R)=\min_{\mathfrak{m} \in S_M}|R / \mathfrak{m}|+1$ holds, a few of which are listed in Section~\ref{S4.2}.

In the final Section~\ref{S5}, we develop the topology on finitely generated modules, to address the covering problem from (to the best of our knowledge) a novel perspective, and in the spirit of Theorem~\ref{introT1}. We will only consider those $R$-modules $M$ for which the residue field $R/\mathfrak{m}$ is infinite whenever $\mathfrak{m}\in S_M$, unless otherwise mentioned. Given such a finitely generated module $M$, we consider $M/\mathfrak{m}M$ (for all $\mathfrak{m}\in \mspec(R)$) with the induced Zariski topology or the discrete subspace topology (defined in Section~\ref{S5}) and topologize $M$ with the coarsest topology making the natural map $q:M\rightarrow \prod_{\mathfrak{m}\in \mspec(R)}M/\mathfrak{m}M$ into a continuous map, where $\prod_{\mathfrak{m}\in \mspec(R)}M/\mathfrak{m}M$ is equipped with the product topology. We shall call this topology to be the \textit{induced Zariski topology} on $M$, as it generalizes the induced Zariski topology on vector spaces.

Define $m(I)$ to be the collection of all maximal ideals $\mathfrak{m}\in\mspec(R)$ containing the ideal $I$ of $R$, for every non-trivial proper ideal $I\subset R$. We topologize $M/IM$ with the coarsest topology which makes the natural map $p_I:M/IM\rightarrow \prod_{\mathfrak{m}\in m(I)}M/\mathfrak{m}M$ continuous, where $\prod_{\mathfrak{m}\in m(I)}M/\mathfrak{m}M$ has the product topology. Let $\mathcal{I}(R)$ denote the poset of all non-trivial proper ideals of $R$ ordered by reverse-inclusion. Then we get an inverse system of spaces $\Sigma_2=((M/IM)_{I\in\mathcal{I}(R)}, (q_{IJ})_{I\preccurlyeq J\in\mathcal{I}(R)})$, where the bonding maps $q_{IJ}$ are the natural continuous (module-)quotient maps. Corresponding to this system, we have an inverse limit space $\varprojlim_{I\in\mathcal{I}(R)}M/IM$. Similarly, considering $\mathcal{F}=\lbrace \sigma\subseteq \mspec(R):$ $|\sigma|<\infty\rbrace$ as a poset ordered by inclusion, and letting $M_\sigma=\prod_{\mathfrak{m}\in \sigma}M/\mathfrak{m}M$, we get another inverse system $\Sigma_1=((M_{\sigma})_{\sigma\in\mathcal{F}}, (p_{ij})_{\sigma_i\preccurlyeq\sigma_j\in\mathcal{F}})$, where the bonding maps $p_{ij}$ are the natural continuous projection maps. Corresponding to this system, we have an inverse limit space $\varprojlim_{\sigma\in\mathcal{F}}M_{\sigma}$. We then have the following homeomorphisms.

\begin{theorem}\label{introT3}
    The topological spaces $\prod_{\mathfrak{m}\in \mspec(R)}M/\mathfrak{m}M$, $\varprojlim_{\sigma\in\mathcal{F}}M_{\sigma}$, and $\varprojlim_{I\in\mathcal{I}(R)}M/IM$ are all homeomorphic via the natural maps.
\end{theorem}

 Theorem~\ref{introT3} shows that the different ways of combining the topology on the ``simpler" factors all lead to the ``same" (i.e. up to homeomorphism) space.  This theorem also enables one to consider the topology as a $T1$ generalization of the notion of the profinite topology, when $S_M:=\lbrace \mathfrak{m}\in \mspec(R):$ $\dim_{R/\mathfrak{m}}(M/\mathfrak{m}M)\geq2\rbrace$ equals $\mspec(R)$. This is because the topologies on the factors $M_\sigma$ degenerate to discrete topologies if $M/\mathfrak{m}M$ are finite spaces for all $\mathfrak{m}\in \mspec(R)$, in which case one obtains the profinite topology.

The next results focus on special cases. We observe that a module $M$ has finite dual Goldie dimension if and only if $\prod_{\mathfrak{m}\in \mspec(R)}M/\mathfrak{m}M$ is a finite product. Consequently, we have:

\begin{theorem}\label{introT4}
   Let $M$ be a finitely generated non-cyclic $R$-module with finite dual Goldie dimension over a commutative ring $R$, such that the residue field $R/\mathfrak{m}$ are infinite for all $\mathfrak{m}\in S_M$. Equip $M$ with the induced Zariski topology as defined above. Then $\sigma_{\tau}(M, R)=\min_{\mathfrak{m}\in S_M}|R/\mathfrak{m}|+1$. Consequently, $(a)$ the induced Zariski topology on $M$ makes it into a $\kappa_M$-maximal Baire space, for $\kappa_M=\min_{\mathfrak{m}\in S_M}|R/\mathfrak{m}|+1 $; and $(b)$ the covering number of $M$ is $\sigma(M, R)=\kappa_M$.
\end{theorem}

Here, we recall the definition of a $\kappa$-Baire and $\kappa$-maximal Baire space.

\begin{definition}\label{introD1}
\begin{enumerate}
    \item  A topological space $X$ is a \textit{$\kappa$-Baire space} if for any collection of dense open subsets $\lbrace \mathcal{O}_i\rbrace_{i\in I}$ of $X$ such that $|I|=\mathfrak{I}<\kappa$, $\cap_{i\in I}\mathcal{O}_i$ is dense in $X$ as well.
    
    \item A topological space $X$, which is a $\kappa$-Baire space, but is not a $\aleph$-Baire space for any cardinal number $\aleph>\kappa$, is said to be a \textit{$\kappa$-maximal Baire space}.
\end{enumerate}
    \end{definition}

\begin{remark}
A `naive' way to discuss the equality $\sigma_\tau(M,R) = \sigma(M,R) =
\min_{\mathfrak{m} \in S_M} |R/\mathfrak{m}|+1$ is as follows: if $M$ is finitely generated and $N \subsetneq M$ is a
proper submodule of $M$, then the closure of $N$ in $M$, under the induced Zariski topology, is a subset of any proper maximal submodule $\tilde{N}$ containing $N$  (since all maximal submodules of $M$ are closed). Therefore, the
inequalities
\[
\sigma_\tau(M,R) \leq \sigma(M,R) \leq \min_{\mathfrak{m} \in S_M}
|R/\mathfrak{m}|+1
\]
are rather straightforward, where $S_M$ is as above. The
contribution in \cite{khare2020} was to show that the second inequality
is in fact an equality for various classes of rings $R$ and $R$-modules
$M$. In this paper, (a) we show this equality holds for somewhat larger classes of
finitely generated modules $M$; and (b) we then show that the
\textit{first} inequality $\sigma_\tau(M,R) \leq \sigma(M,R)$ is also an
equality, under reasonable technical assumptions.
\end{remark}

Returning to the above discussion, we next use the characterization of $\kappa$-Baire spaces as those which are not of first $\kappa$-category (Definition~\ref{D7}), and obtain more generally:

\begin{proposition}\label{introP1}
     	Let $M$ be a finitely generated $R$-module with the induced Zariski topology, and let $\kappa_M=\min_{\mathfrak{m}\in S_M}|R/\mathfrak{m}|+1$. If the image of $M$ under the map $q: M\rightarrow\prod_{\mathfrak{m}\in \mspec(R)}M/\mathfrak{m}M$ is not of first $\kappa_M$-category, then $\sigma_{\tau}(M, R)=\kappa_M$. In particular, $\sigma(M, R)= \kappa_M$.
\end{proposition}

Noting that $\prod_{\mathfrak{m}\in m(I)}M/\mathfrak{m}M$ is a $\kappa_M$-Baire space, and since we are interested in $q(M)$ when it is $\kappa_M$-Baire itself, we obtain a characterization of surjectivity of $q$ in a special case:

\begin{proposition}\label{introP2}
    	If $q(M)=\prod_{\mathfrak{m}\in \mspec(R)}M/\mathfrak{m}M$, i.e., $q$ is surjective, then $M$ is compact. Under the assumption that the set $S_M$ equals $\mspec(R)$, the converse holds: if $M$ is a compact space under its induced Zariski topology, then $q(M)=\prod_{\mathfrak{m}\in \mspec(R)}M/\mathfrak{m}M$, i.e., $q$ is surjective. 
\end{proposition}

The topology we define on a finitely generated $R$-module $M$, to some extent, behaves better under the assumption $S_M=\lbrace \mathfrak{m}\in \mspec(R):$ $\dim_{R/\mathfrak{m}}(M/\mathfrak{m}M)\geq2\rbrace$ equals $\mspec(R)$. Letting $T:=R\setminus\bigcup_{\mathfrak{m}\in S_M}\mathfrak{m}$, one can consider the $T^{-1}R$-module $T^{-1}M$, i.e. localization of $M$ at the multiplicative set $T$. Then one naturally has the following equality of sets:
$$S_{T^{-1}M}:=\lbrace T^{-1}\mathfrak{m}\in \mspec(T^{-1}R): \ \dim_{(T^{-1}R/T^{-1}\mathfrak{m})}(T^{-1}M/T^{-1}(\mathfrak{m}M))\geq2\rbrace=\mspec(T^{-1}R).$$
Topologizing $T^{-1}M$ with the induced Zariski topology, as a module over $T^{-1}R$, we can then relate $T^{-1}M$ and $M$ as topological spaces:

\begin{proposition}\label{introP3}
  Let $M$ be a finitely generated $R$-module, topologized with the induced Zariski topology. Let $S_M=\lbrace \mathfrak{m}\in \mspec(R):$ $\dim_{R/\mathfrak{m}}(M/\mathfrak{m}M)\geq2\rbrace$ and $T=R\setminus\bigcup_{\mathfrak{m}\in S_M}\mathfrak{m}$. If $M$ is a $\kappa$-Baire space with its induced Zariski topology (for any infinite cardinal $\kappa$), then the $T^{-1}R$-module $T^{-1}M$, with its induced Zariski topology is also a $\kappa$-Baire space.
\end{proposition}

This result has a corollary that demonstrates its significance for Question~\ref{Q2}.

\begin{corollary}
Let $M$ be a finitely generated $R$-module (with the induced Zariski topology), which satisfies the equality $\sigma_\tau(M, R)=\sigma(M, R)=\min_{\mathfrak{m}\in S_M}|R/\mathfrak{m}|+1$. Letting $T=R\setminus\bigcup_{\mathfrak{m}\in S_M}\mathfrak{m}$, we have that the $T^{-1}R$-module $T^{-1}M$ (with the induced Zariski topology) also satisfies the equality $$\sigma_\tau(T^{-1}M, T^{-1}R)=\sigma(T^{-1}M, T^{-1}R)=\min_{T^{-1}\mathfrak{m}\in \mspec(T^{-1}R)}|T^{-1}R/T^{-1}\mathfrak{m}|+1 = \min_{\mathfrak{m}\in S_M}|R/\mathfrak{m}|+1.$$
\end{corollary}

 We conclude this paper by characterizing when $M$ has a finite Zariski covering number -- see Proposition~\ref{P19}.

\section{Covering problem for vector spaces}\label{S3}

We first briefly review the known results in the well-studied covering problem for vector spaces $V$, and provide an alternate approach to the proofs, via induction on the dimension of $V$. We then provide a topological perspective to the covering problem for vector spaces over infinite fields.

\subsection{Review of known results}\label{S3.1}

We start by revisiting a well-known result for vector spaces. 

\begin{proposition}\label{P1}
	Let $V$ be any vector space over an infinite field $\mathbf{K}$. Then $V$ is not the union of a finite number of proper subspaces.
\end{proposition} 

The standard proof in the literature (for example, Proposition~$2.58$ of \cite{galois})  proceeds via induction on the number of subspaces in the union. We provide two other proofs of Proposition~\ref{P1}-- one presently using linear algebra, and one in Corollary~\ref{C5} below using the (induced) Zariski topology. For the first proof, note that the statement is trivially true for a $1$-dimensional vector space $V$, so we will prove it for vector spaces of dimension at least $2$. We first prove a special case.

\begin{lemma}\label{L1}
	If $V$ is a vector space of dimension $k<\infty$ over an infinite field $\mathbf{K}$, then $V$ is not the union of a finite number of proper subspaces of dimension $k-1$. 
\end{lemma}

\begin{proof}
	First, we prove the lemma for $k=2$ by contradiction. Assume $V$= $\bigcup_{j=1}^{p}(v_j)$, where $(v_j)$ are distinct 1-dimensional subspaces generated by $v_j$ $\in$ $V$, and $p$ $\in$ $\mathbb{N}$. Let $\lbrace w_1, w_2 \rbrace$ be a basis for $V$, and write $v_i = a_{1i} w_1 + a_{2i} w_2$ for all $i$, where $a_{1i}, a_{2i} \in \mathbf{K}$. As $V = \bigcup_{j=1}^p (v_j)$, for each $(\lambda_1, \lambda_2) \in \mathbf{K}^2$ there exists $\beta$ such that $(\lambda_1, \lambda_2) = \beta (a_{1i}, a_{2i})$ for some $i \in \{ 1, 2, \dots, p \}$. Choose $\lambda \in \mathbf{K} \setminus \{ 0, a_{11}, \dots, a_{1p} \}$, and let $Z := \{ i : a_{1i} \neq 0 \}$. As $Z$ is finite, there exists $\gamma \in \mathbf{K} \setminus \{ 0 \}$ such that $\gamma \neq \lambda a_{1i}^{-1} a_{2i}, \ \forall i \in Z$. Now $(\lambda, \gamma) \in \mathbf{K}^2$ is not of the form $\beta (a_{1i}, a_{2i})$ for any $\beta \in \mathbf{K}$ and $1 \leq i \leq p$. Thus $w = \lambda w_1+\gamma w_2$ $\notin$ $\bigcup_{j=1}^{p}(v_j)$. This contradicts  our assumption, proving the lemma for $k=2$.
	
	Now let $k\geq 2$ and assume that the result holds for all vector spaces of dimension $k$ over an infinite field $\mathbf{K}$. Let $V$ be a vector space over $\mathbf{K}$ with $\dim V=k+1$, and fix a subspace $\Omega$ of dimension $k-1$. Now the $k$-dimensional subspaces of $V$ containing $\Omega$ are in $1-1$ correspondence with the $1$-dimensional subspaces of the $2$-dimensional quotient $V/\Omega\cong \mathbf{K}^2$. By the $k=2$ case, there are infinitely many distinct 1-dimensional subspaces in $V/\Omega$. For distinct $1$-dimensional subspaces of $V/\Omega$, we get distinct $k$-dimensional subspaces of $V$. Assume that $V$ is a finite union of $k$-dimensional subspaces $\lbrace V_i\rbrace_{i=1}^{n}$. Then choose a $k$-dimensional subspace $W$ of $V$ distinct from all $V_i$. Then $W=\bigcup_{i=1}^{n}W\cap V_i$, which contradicts the induction hypothesis. 
\end{proof}

This completes the proof of Lemma~\ref{L1}. We now use this lemma to show:

\begin{proof}[Proof of Proposition~\ref{P1}]
	Let $V$ be a finite dimensional vector space of dimension $k>1$ over an infinite field $\mathbf{K}$. Consider any finite union $\bigcup_{i=1}^{n}V_i$ of proper subspaces of $V$. Then each proper subspace $V_i$ is contained in a maximal subspace $\tilde{V}_i$ of $V$, which are also $k-1$ dimensional subspaces of $V$. Thus, $\bigcup_{i=1}^{n}V_i\subseteq \bigcup_{i=1}^{n}\tilde{V}_i\subsetneq V$, where the last strict containment follows from Lemma~\ref{L1} above. This shows the result for finite dimensional vector spaces.
	
	Now let $V$ be an arbitrary infinite dimensional vector space over an infinite field $\mathbf{K}$. Suppose $V$ could be expressed as a finite union of proper subspaces $W_i$ ($i$ = $1, 2, 3, \dots, n$). Choose vectors $w_i\in V\setminus W_i$, for all $i$ and define $X$ to be the (finite-dimensional) span of $\lbrace w_1, w_2, \dots, w_n \rbrace$. Then the union of $X\cap W_i$  over all $i$ $\in$ $\lbrace 1, 2, \dots, n\rbrace$ covers $X$. Also $X\cap W_i\subsetneq X$, since $w_i\in X\setminus W_i$. However, this contradicts Lemma~\ref{L1}.
\end{proof}

In the proof of Lemma~\ref{L1}, we first dealt with the case $k=2$ and then used this statement to conclude the presence of enough $k-1$ dimensional subspaces in the case of general finite $k$, to use the induction hypothesis. The reduction process via quotienting did not make use of the hypothesis that $\mathbf{K}$ is an infinite field. Hence, one can apply the same reduction to the statement when $\mathbf{K}$ is a finite field or an infinite field of arbitrary cardinality, and explicitly evaluate the minimum (cardinal) number of proper subspaces required to cover a given finite-dimensional vector space.

\begin{definition}\label{D1}
    For any vector space $V$ over any base field $\mathbf{K}$, the minimum (cardinal) number of proper subspaces required to cover $V$ is called the \textit{covering number} of $V$ and is denoted by $\sigma (V, \mathbf{K})$.
\end{definition}

The evaluation of $\sigma(V, \mathbf{K})$ was discussed in \cite{KHARE20091681}. Hence, we skip the proofs and only sketch our arguments in this subsection. Using a simple counting argument, we first have:

\begin{lemma}\label{L2}
	Let $V$ be a $2$-dimensional vector space over $\mathbf{F}_{p^n}$. Then $V$ can be expressed as a union of exactly $p^n+1$ proper subspaces. In particular $\sigma(V, \mathbf{F}_{p^n})=|\mathbf{F}_{p^n}|+1$, for $2$-dimensional $V$.
\end{lemma}

  Lemma~\ref{L2} helps prove the following result, along the lines of the proof of Proposition~\ref{P1}, by first proving the statement for finite-dimensional vector spaces $V$ using Lemma~\ref{L2} and induction on dimension. The infinite-dimensional case then follows, as in the proof of Proposition~\ref{P1}.

\begin{proposition}\label{P3}
	Let $V$ be a vector space over a finite field $\mathbf{F}_{p^n}$ of order $p^n$, with dimension greater than $1$ (maybe infinite). Then $V$ is not the union of any collection of at most $p^n$ proper subspaces. Thus, $\sigma(V, \mathbf{F}_{p^n})=|\mathbf{F}_{p^n}|+1$, for any vector space $V$. 
\end{proposition}

These methods work equally well for proving the analogous results for the case of vector spaces over infinite fields of arbitrary cardinalities. First, we note that the proof of Lemma~\ref{L1} can be replicated to prove the following lemma.

\begin{lemma}\label{P5}
	If $V$ is a $k$-dimensional ($k>1$ and finite) vector space over an infinite field $\mathbf{K}$ of cardinality $|K| = \kappa$, then $V$ is not the union of $\mathfrak{I}$-many proper subspaces of dimension $k-1$, if $\mathfrak{I} < \kappa$ ($\mathfrak{I}$ is a cardinal number). 
\end{lemma} 

Lemma~\ref{P5} implies the next result, which is a generalization of the finite-dimensional case in Proposition~\ref{P1}. Since the proof is almost identical to that of Proposition~\ref{P1}, it is omitted.

\begin{proposition}\label{P6}
	Let $V$ be a finite-dimensional vector space over an infinite field $\mathbf{K}$ of cardinality $|K|=\kappa$. Then $V$ is not the union of $\mathfrak{I}$-many proper subspaces of $V$ if $\mathfrak{I}<\kappa$.
\end{proposition}

 Note that in the above case of finite-dimensional vector spaces $V$ over a base field of infinite cardinality $\kappa$, one can express $V$ as a union of $\kappa=\kappa+1$ many proper subspaces, when the dimension of $V$ is at least $2$. Thus, the results so far can be summarized as:

\begin{theorem}\label{T1}
	For any finite-dimensional vector space $V$ with dimension at least $2$, over any base field $\mathbf{K}$, we have $\sigma (V, \mathbf{K})=|\mathbf{K}|+1$. For finite base fields $\mathbf{K}$, we have $\sigma(V, \mathbf{K})=|\mathbf{K}|+1$, for any (even infinite-dimensional) vector space $V$ over $\mathbf{K}$, of dimension at least $2$.
\end{theorem}

For completeness, we remark that for an infinite-dimensional vector space $V$ over an infinite base field $\mathbf{K}$, one always has $\sigma(V, \mathbf{K})=\aleph_0=|\mathbb{N}|$. Hence, Theorem~\ref{T1} cannot be extended to such cases in general, unless the base field is countable.

\subsection{A topological perspective}\label{S3.4}

This section revisits Proposition~\ref{P1} from a topological point of view. Let $V$ be a $n$-dimensional vector space over an infinite field $\mathbf{K}$, for some natural number $n$. Fix an ordered basis $\mathcal{B} =  (e_1, \dots, e_n)$ of $V$. This induces an isomorphism $\tau : V \rightarrow \mathbf{K}^n$, mapping $\sum_{i=1}^{n}a_i\cdot e_i \mapsto(a_1, a_2, \dots, a_n)$. 

Consider the Zariski topology on the affine space $\mathbf{K}^n$, that is, the closed subsets of $\mathbf{K}^n$ are precisely the algebraic sets $V(\mathcal{F})$, $\mathcal{F}\subset \mathbf{K}[X_1, X_2, \dots, X_n]$ defined as follows:
$$V(\mathcal{F}) := \lbrace (a_1, \dots, a_n)\in \mathbf{K}^n | f(a_1, \dots, a_n)= 0 \ \forall f(X_1, \dots, X_n) \in \mathcal{F}\rbrace$$

We make $V$ into a topological space by equipping it with the coarsest topology such that the isomorphism $\tau: V\rightarrow \mathbf{K}^n$ becomes a homeomorphism: that is, define the closed sets in $V$ as precisely the sets $\tau^{-1}(V(\mathcal{F})) \subset V$. We shall call this the \textit{induced Zariski topology} on $V$.

Next, we claim that \textit{the induced Zariski topology on the vector space $V$ is canonical}, i.e., it is independent of the choice of basis of $V$ made to get an isomorphism with $\mathbf{K}^n$:

To see why, let $\mathcal{B}_1=( v_1, \dots, v_n)$, $\mathcal{B}_2=( w_1, \dots, w_n)$ be two ordered bases of $V$. Let $C = (c_{ij})_{i,j=1}^n$ and $D = (d_{ij})_{i,j=1}^n$ denote the change-of-basis matrices from $\mathcal{B}_1$ to $\mathcal{B}_2$ and vice versa -- that is, $w_j = \sum_{i=1}^n c_{ij} v_i$ and $v_j = \sum_{i=1}^n d_{ij} w_i$. Now define the map $\Psi_D : \mathbf{K}[X_1, \dots, X_n] \to \mathbf{K}[X_1, \dots, X_n]$ via: $\Psi_D(f)(X_1, \dots, X_n) := f(Y_1, \dots, Y_n)$, where the transformed variables $Y_i = \sum_{j=1}^n d_{ij} X_j$. Similarly define $\Psi_C$; it is then clear that $\Psi_C, \Psi_D$ are mutually inverse algebra isomorphisms of $\mathbf{K}[X_1, \dots, X_n]$. In particular, $\tau_2^{-1}(V(\mathcal{F})) = \tau_1^{-1}(V(\Psi_D(\mathcal{F}))$ for every subset $\mathcal{F} \subset \mathbf{K}[X_1, \dots, X_n]$. Thus, the induced Zariski topology is independent of the choice of basis of $V$. \qed \medskip

Note that for any subset $\mathcal{F}\subset \mathbf{K}^n$, $V(\mathcal{F}) = V((\mathcal{F}))$, where $(\mathcal{F})$ is the ideal generated by $\mathcal{F}$ in $\mathbf{K}[X_1, X_2, \dots, X_n]$. So from here on, we only consider closed sets of the form $V(\mathcal{I})$ for ideals $\mathcal{I}\subset\mathbf{K}[X_1, X_2, \dots, X_n]$. It is routine to verify that the Zariski topology is indeed a well-defined topology on the affine space $\mathbf{K}^n$, so the topology induced by the isomorphism $\tau$ on $V$ from the Zariski topology on $\mathbf{K}^n$ is also well-defined.

\begin{definition}\label{D2}
	A topological space $X$ is said to be \textit{irreducible} if $X\neq \emptyset$ and if every pair of open sets in $X$ intersect. An equivalent definition is that every non-empty open subset in $X$ is dense in $X$. Another equivalent definition is that $X$ is not the union of any two proper closed subsets of $X$.
\end{definition}

\begin{remark}\label{R6}
	 A closed irreducible subspace $A$ of a topological space $X$ is a closed subset of $X$ which is irreducible as a topological space under the subspace topology. Such a subset $A$ is not the union of any finite number of proper closed subspaces of $A$. This is largely reminiscent of the results obtained in the case of vector spaces, previously. We will show, this is not a coincidence.
\end{remark}

Now we see a sketch of the fact that if $\mathbf{K}$ is an infinite field, then any finite-dimensional $\mathbf{K}$-vector space $V$ is an irreducible topological space under the induced Zariski Topology. This is in fact stronger than Proposition~\ref{P1}, which will follow as a special case of this, as it is easy to see that vector subspaces of $V$ are closed subsets under the induced Zariski topology.

To each closed subset $V(\mathcal{F}) \subset \mathbf{K}^n$,  associate its \textit{vanishing ideal}, given by:
\[
\mathcal{I}(V(\mathcal{F})):=\lbrace f\in\mathbf{K}[X_1, \dots, X_n]: f \text{ vanishes on } V(\mathcal{F})\rbrace
\]
The following well-known proposition characterizes the vanishing ideals $\mathcal{I}(V(\mathcal{F}))$ of $\mathbf{K}[X_1, \dots, X_n]$, for which the closed subset $V(\mathcal{F})$ is a closed irreducible subspace of the Zariski topology on $\mathbf{K}^n$:

\begin{proposition}[Proposition $1$, Section $1.5$ in \cite{fulton}]\label{P7}
	Let $\mathbf{K}$ be a field, $n\in\mathbb{N}$, and let $V(\mathcal{F})\subset \mathbf{K}^n$ be a Zariski closed subset. Then the following are equivalent:
	\begin{enumerate}
		\item $V(\mathcal{F})$ is a closed irreducible subset.
		\item The vanishing ideal $\mathcal{I}(V(\mathcal{F}))$ is a prime ideal of $\mathbf{K}[X_1, X_2, \dots, X_n]$.
	\end{enumerate}
\end{proposition}

Now consider the subset $\lbrace0\rbrace=\mathcal{F}\subset \mathbf{K}[X_1, X_2, \dots, X_n]$. Note that $V(\mathcal{F}) = \mathbf{K}^n$. Thus, $\mathbf{K}^n$ is a closed set under the Zariski topology. The following lemma is standard, see e.g. \cite[Exercise $1.4$]{fulton}. 

\begin{lemma}\label{L3}
	Let $\mathbf{K}$ be an infinite field. Let $n\in\mathbb{N}$.  For the subset $\lbrace0\rbrace=\mathcal{F}\subset \mathbf{K}[X_1, X_2, \dots, X_n]$, the vanishing ideal $\mathcal{I}(V(\mathcal{F}))$ is the prime ideal $(0)$. Thus, the affine space $\mathbf{K}^n$ with the Zariski topology is a closed irreducible topological space.
\end{lemma}

Since a finite-dimensional $\mathbf{K}$-vector space $V$ under the induced Zariski topology is homeomorphic to an affine space $\mathbf{K}^n$ under the Zariski topology, Lemma~\ref{L3} implies that,

\begin{corollary}\label{C5}
	A finite-dimensional vector space $V$ over an infinite field $\mathbf{K}$ is an irreducible topological space under the induced Zariski topology. In particular, Proposition~\ref{P1} holds.
\end{corollary}

The final sentence is true because vector subspaces of $V$ are closed in the induced Zariski topology.

\subsubsection{Further generalizations}\label{S3.4.2}

We saw that vector subspaces of affine space $\mathbf{K}^n$ are algebraic sets. Till now, we have been seeing various ways to understand how vector subspaces may cover a finite-dimensional vector space $V$ over an infinite field $\mathbf{K}$, by equipping $V$ with the induced  Zariski topology. We now extend these considerations to covering $V$ by algebraic sets. We will assume that the cardinality of the base field $\mathbf{K}$ is $\kappa$, where $\kappa$ is an infinite cardinal.

We first note that the affine space $\mathbf{K}^1$ is not a union of $\mathfrak{I}$-many proper algebraic sets, for any cardinal $\mathfrak{I}<\kappa$. This is because a proper algebraic set $V(\mathcal{F})$ is the set of common roots of some collection of polynomials $\mathcal{F}\subseteq \mathbf{K}[X]$, such that $\mathcal{F}\neq\lbrace 0\rbrace$. Since there exists a non-zero polynomial $f\in \mathcal{F}$, it follows that $V(\mathcal{F})$ is finite as a set. Clearly, then a union of $\mathfrak{I}$-many proper algebraic sets ($\mathfrak{I}<\kappa$) has cardinality less than $\kappa$, and hence cannot cover the entire affine space $\mathbf{K}^1$. We claim this is true for any affine space $\mathbf{K}^m$ over an infinite field:

\begin{theorem}\label{T2}
	Let $\mathbf{K}$ be a field of cardinality $\kappa$, an infinite cardinal. Then the affine space $\mathbf{K}^m$ cannot be written as a union of $\mathfrak{I}$-many proper algebraic sets, for any cardinal $\mathfrak{I}<\kappa$ and $m\geq1$.
\end{theorem}

\begin{proof}
	The $m=1$ case was shown above. Assume the result is true for $\mathbf{K}^{m-1}$ for some $m\in\mathbb{N}$, $m>1$. Suppose $\mathbf{K}^m$ can be covered by a union of $\mathfrak{I}$-many proper algebraic sets $\lbrace V_i\rbrace_{i\in \mathcal{I}}$, for some index set $\mathcal{I}$ of cardinality $\mathfrak{I}<\kappa$. Note that $V_i=V(J_i)$, for some ideal $\lbrace0\rbrace\subsetneq J_i\subset \mathbf{K}[X_1, X_2, \dots, X_m]$, $\forall i\in \mathcal{I}$. Any Zariski closed subspace of $\mathbf{K}^m$ given by $V((a_0+a_1X_1+a_2X_2+\dots+a_m X_m))$ is a hyperplane in $\mathbf{K}^m$. Clearly there are at least $\kappa$-many disjoint hyperplanes in $\mathbf{K}^m$ given by $X_m=a$, for all $a\in \mathbf{K}$, if $|\mathbf{K}|=\kappa$. Any hyperplane $X=V((a_0+a_1X_1+a_2X_2+\dots+a_mX_m))$ is a homeomorphic copy of $\mathbf{K}^{m-1}$ in $\mathbf{K}^m$. Since $\lbrace V_i\rbrace_{i\in\mathcal{I}}$ covers $\mathbf{K}^m$, it follows that $\lbrace V_i\cap X\rbrace_{i\in\mathcal{I}}$ covers $X$. By the induction hypothesis, this is not possible unless $X\subseteq V_i$ for some $i\in\mathcal{I}$. Thus, if $\lbrace V_i\rbrace_{i\in\mathcal{I}}$ covers $\mathbf{K}^m$, then each hyperplane is contained in some $V_i$, $i\in\mathcal{I}$.
	
	Let $V_0$ be a particular algebraic set in the collection $\lbrace V_i\rbrace_{i\in\mathcal{I}}$. Let $\lbrace H_{j}\rbrace_{j\in\mathcal{F}}$ be the collection of distinct hyperplanes in $\mathbf{K}^{m}$, contained in $V_0$. Let $H_{j}=V(f_{j}(X_1, X_2, \dots, X_m))$, where $f_{j}$ is a linear polynomial, $\forall j\in \mathcal{F}$. Consider a particular hyperplane $H_{j_0}=V(f_{j_0})$ in the collection, where $f_{j_0}(X_1, X_2, \dots, X_m)= a_0+a_1X_1+a_2X_2+\dots+a_mX_m$ and, without loss of generality, assume $a_m=1$. Then $f_{j_0}$ is a primitive polynomial in $\mathbf{K}[X_1, X_2, \dots, X_{m-1}][X_m]$. Since it is linear in $X_m$, $f_{j_0}$ is a linear polynomial in $\mathbf{K}(X_1, X_2, \dots, X_{m-1})[X_m]$, and is hence irreducible in $\mathbf{K}(X_1, X_2, \dots, X_{m-1})[X_m]$. By Gauss's lemma, $f_{j_0}$ is irreducible in $\mathbf{K}[X_1, X_2, \dots, X_{m-1}][X_m]$, i.e. in $\mathbf{K}[X_1, X_2, \dots, X_m]$. Moreover, $j_0\in\mathcal{F}$ was arbitrary.
	
	Thus, for the collection $\lbrace H_{j}\rbrace_{j\in\mathcal{F}}$ of distinct hyperplanes in $\mathbf{K}^{m}$ contained in $V_0$, let $H_{j}=V(f_{j}(X_1, X_2, \dots, X_m))$, where $f_{j}$ is a linear polynomial (and hence irreducible in $\mathbf{K}[X_1, \dots, X_m]$), $\forall j\in \mathcal{F}$. Let $J_0$ be the vanishing ideal of $V_0$ in $\mathbf{K}[X_1, \dots, X_m]$, i.e $V_0=V(J_0)$. Since $\bigcup_{j\in\mathcal{F}} H_{j}\subseteq V_0$, it follows that $J_0\subseteq \bigcap_{j\in\mathcal{F}}(f_{j})$. We claim that $\bigcap_{j\in\mathcal{F}}(f_{j})$ is non-trivial only if the index set $\mathcal{F}$ is a finite set. This is because if $0\neq g\in \bigcap_{j\in\mathcal{F}}(f_{j})$, then $f_{j}|g$ for all $j\in\mathcal{F}$, and $f_{j}$ are all distinct prime elements, none of which are associates of each other (since they correspond to distinct hyperplanes). Since $\mathbf{K}[X_1, X_2, \dots, X_m]$ is a UFD, $\mathcal{F}$ must be finite. Since $V_0\in\lbrace V_i\rbrace_{i\in\mathcal{I}}$ is arbitrary, it follows that any algebraic set $V_i$ can contain only finitely many distinct hyperplanes.
	
	Thus, the collection $\lbrace V_i\rbrace_{i\in\mathcal{I}}$ can cover finitely many distinct hyperplanes if $|\mathcal{I}|=\mathfrak{I}$ is finite and at most $\mathfrak{I}$ many distinct hyperplanes, if $\mathfrak{I}$ is infinite. Hence, if $\mathfrak{I}<\kappa=|\mathbf{K}|$, then the collection $\lbrace V_i\rbrace_{i\in\mathcal{I}}$ cannot even cover all the hyperplanes in $\mathbf{K}^m$, and hence cannot cover the entire affine space $\mathbf{K}^m$. We are done by induction on the dimension of the affine space.
\end{proof}

The above result also holds for $m$-dimensional vector spaces over $\mathbf{K}$, with the induced Zariski topology, in light of the homeomorphism with $\mathbf{K}^m$. Thus, Theorem~\ref{T2} generalizes Proposition~\ref{P6} to cover by algebraic sets. One can interpret the above result, as a strengthening of the irreducibility of the Zariski topology, as follows. 

\begin{definition}\label{D3}
      For a finite-dimensional vector space $V$ over an infinite base field $\mathbf{K}$, the minimum (cardinal) number of proper closed subsets of $V$ with the induced Zariski topology over $\mathbf{K}$, whose union covers $V$, is said to be the \textit{Zariski covering number} of $V$ and is denoted by $\sigma_{\tau}(V, \mathbf{K})$.
\end{definition}

Using the above definition,  we see that for an infinite field $\mathbf{K}$, the irreducibility of the topological space $V$ is equivalent to the inequality $\sigma_{\tau}(V, \mathbf{K})\geq \aleph_0$. Theorem~\ref{T2} strengthens this inequality, to the equality $\sigma_{\tau}(V, \mathbf{K})=|\mathbf{K}|+1$, for any finite-dimensional vector space $V$ with the induced Zariski topology over an infinite base field $\mathbf{K}$. For finite-dimensional vector spaces $V$ with $\dim V\geq2$, one also has the equality $\sigma_{\tau}(V, \mathbf{K})=\sigma(V, \mathbf{K})$, when considering $V$ with the induced Zariski topology, over an infinite base field.  The notation suppresses the vector-space dimension $m$ since clearly, both these quantities are independent of it. The significance of this equality is that it allows us to interpret $\sigma (V, \mathbf{K})$, an algebraic quantity, as a topological quantity.

Now we will interpret Theorem~\ref{T2} as a generalization of  the Baire Category Theorem.

\begin{definition}\label{D4}
\begin{enumerate}
    \item A topological space $X$ is a \textit{$\kappa$-Baire space} if for any collection of dense open subsets $\lbrace \mathcal{O}_i\rbrace_{i\in I}$ of $X$ such that $|I|=\mathfrak{I}<\kappa$, $\cap_{i\in I}\mathcal{O}_i$ is dense in $X$ as well.
    
    \item A topological space $X$, which is a $\kappa$-Baire space, but is not a $\aleph$-Baire space for any cardinal number $\aleph>\kappa$, is said to be a \textit{$\kappa$-maximal Baire space}.
\end{enumerate}
	 
\end{definition}

Assuming the continuum hypothesis, an ordinary Baire space in the literature is an $\aleph_1$- Baire space, by our definition.

\begin{corollary}
	Let $\mathbf{K}$ be a field of cardinality $\kappa$, an infinite cardinal. Then the affine space $\mathbf{K}^n$ is a $\kappa$-maximal Baire space. 
\end{corollary}

\begin{proof}
	Let $\lbrace \mathcal{O}_i\rbrace_{i\in I}$ be a collection of (dense) open subsets of affine space $\mathbf{K}^n$ such that $|I|=\mathfrak{I}<\kappa$. Then $\cap_{i\in I}\mathcal{O}_i\neq \emptyset$, since by Theorem~\ref{T2}, $\bigcup_{i\in I}(\mathbf{K}^n\setminus \mathcal{O}_i)\neq \mathbf{K}^n$. Let $x\in \mathbf{K}^n\setminus \cap_{i\in I}\mathcal{O}_i$, and let $\mathcal{U}$ be any open neighborhood of $x$. Then $\mathcal{U}\cap \bigcap_{i\in I}\mathcal{O}_i\neq \emptyset $, since $(\mathbf{K}^n\setminus\mathcal{U})\cup\bigcup_{i\in I}(\mathbf{K}^n\setminus \mathcal{O}_i)\neq \mathbf{K}^n$, by Theorem~\ref{T2}. Thus $x$ is a limit point of $\cap_{i\in I}\mathcal{O}_i$. Since $x$ is arbitrary, it follows that $\cap_{i\in I}\mathcal{O}_i$ is dense in affine space $\mathbf{K}^n$, thereby proving that $\mathbf{K}^n$ is a $\kappa$-Baire space. Clearly $|\mathbf{K}^n|=\kappa$, and thus $\bigcap_{x\in\mathbf{K}^n}(\mathbf{K}^n\setminus\lbrace x\rbrace)=\emptyset$ is not dense in $\mathbf{K}^n$. Hence, $\mathbf{K}^n$ is a $\kappa$-maximal Baire space.
\end{proof}

\section{Covering problem for modules}\label{S4}

This section studies the results of the previous section in greater generality. We start by considering a certain class of modules over rings which naturally generalize vector spaces over fields, namely semi-simple modules. \textbf{All rings considered here are unital and commutative}.

\subsection{Motivation for semi-simple modules}\label{S4.1}

We first quickly review two key points in the proof of Proposition~\ref{P1}, via Lemma~\ref{L1}. Note that a vector space $V$ with a basis $\lbrace e_1, e_2, \dots, e_n\rbrace$ over a field $\mathbf{K}$ can be written as a direct sum of $1$-dimensional subspaces: $V=\mathbf{K}e_1\oplus \mathbf{K}e_2\oplus \dots \oplus \mathbf{K}e_n$. Here, $(a)$ the  direct summands are all simple/irreducible $\mathbf{K}$-modules and $(b)$ there exists a dimension function $\dim:Vec_{\mathbf{K}} \rightarrow \mathbb{N}$, which is additive over short exact sequences in $Vec_{\mathbf{K}}$. These properties admit extensions to the category of finitely generated semi-simple $R$-modules, where the notion of \textit{length} plays the role of dimension. We quickly recall some basics, for completeness.

\begin{definition}[Length]\label{D5}
	 Let $R$ be a ring and $M$ an $R$-module. We call $M\neq0$ simple if its only proper submodule is $0$. We call a chain of submodules, $0=M_0\subsetneq M_1\subsetneq M_2 \subsetneq\dots\subsetneq M_m=M$ a composition series of length $m$ if each successive quotient $M_i/M_{i-1}$ is simple. Define the length $\len(M)$ of $M$ to be the length of any composition series of $M$. By convention, if $M$ has no (finite) composition series, one defines $\len(M) := \infty$. Further $\len(M)=0$ if and only if $M=0$. 
\end{definition}

The length is well-defined (i.e. independent of the choice of the composition series of $M$) by the Jordan-Hölder theorem. Also, for finitely generated semi-simple modules, the length of the module equals the number of summands in its decomposition into the direct sum of simple modules.

\subsection{The covering number of modules with finite dual Goldie dimension}\label{S4.2}

We first recall the quantity of interest in this section.

\begin{definition}
    The \textit{covering number} of a given $R$-module $M$, denoted by $\sigma(M, R)$, is the minimum (cardinal) number of proper $R$-submodules whose union equals $M$.
\end{definition}

Throughout this section, given a finitely generated $R$-module $M$, $S_M$ will denote the set $\lbrace \mathfrak{m}\in\mspec(R):$ $\dim_{R/\mathfrak{m}}(M/\mathfrak{m}M)\geq 2\rbrace$. It is easy to see via Theorem~\ref{T1} that for any module $M$ over a commutative ring $R$, $\sigma(M, R)\leq \min_{\mathfrak{m}\in S_M}|R/\mathfrak{m}|+1$, as shown in \cite{khare2020}. In the same paper, the authors prove that the equality $\sigma(M, R)=\min_{\mathfrak{m}\in S_M}|R/\mathfrak{m}|+1$ holds for some classes of modules, a few of which are: $(a)$ modules $M$ admitting a finite covering, $(b)$ finitely generated modules $M$ over quasi-local rings, and $(c)$ finitely generated torsion modules over Dedekind domains.

In this section, we will prove that the equality $\sigma(M, R)=\min_{\mathfrak{m}\in S_M}|R/\mathfrak{m}|+1$ also holds for modules with finite dual Goldie dimension (this includes cases $(b)$ and $(c)$ in the preceding paragraph). One can refer to \cite{dgdi} for the definition of the dual Goldie dimension, along with some of its properties. Although the authors of \cite{dgdi} use the term ``corank" and the notation $corank(\cdot)$ for dual Goldie dimension, we will use the relatively modern notation $\hdim(\cdot)$, as in \cite{lomp}, since the dual Goldie dimension is also known as the \textit{hollow dimension}. However, we first show that the equality $\sigma(M, R)=\min_{\mathfrak{m}\in S_M}|R/\mathfrak{m}|+1$ holds for semi-simple modules with finite length, which will be crucial in proving the equality for the general case.

\begin{proposition}\label{P12}
	Let $M$ be a finite length semi-simple $R$-module over a ring $R$. If $S_M=\emptyset$, then $M$ is cyclic, and hence cannot be covered by proper submodules. Else $\sigma(M, R)=\min_{\mathfrak{m}\in S_M}|R/\mathfrak{m}|+1$.
\end{proposition}

In fact this holds even without assuming semi-simplicity; see Theorem~\ref{T5}.

\begin{proof}
	For a finite length semi-simple module $M$ over $R$ with $\len(M)=k$, write $M\cong \bigoplus_{i=1}^{r}(R/\mathfrak{m}_i)^{k_i}$, where $\mathfrak{m}_i$ are pairwise distinct maximal ideals of $R$ for $i=1, \dots, r$ and $k_1, \dots, k_r$ are positive integers that sum to $k$. Clearly if $S_M=\emptyset$, $k_i=1$ for all $i=1, \dots, r$, and it follows by the Chinese remainder theorem that $M$ is cyclic, and hence not a union of proper submodules.
	
	Now assume $S_M\neq\emptyset$. We know that for any maximal ideal $\mathfrak{m}$, $\mathfrak{m}(R/\mathfrak{m_i})$ is $0$ if $\mathfrak{m}=\mathfrak{m}_i$ and is $R/\mathfrak{m}_i$ if $\mathfrak{m}\neq\mathfrak{m_i}$. Thus, $M/\mathfrak{m}_iM\cong (R/\mathfrak{m}_i)^{k_i}$ for $i=1, \dots, r$, and $M/\mathfrak{m}M=0$  if $\mathfrak{m}\neq \mathfrak{m}_i$ for any $i$.
	
	We proceed by induction on length. Let $M$ be a semi-simple module over $R$ of length $2$ with $S_M\neq\emptyset$. This implies $M\cong R/\mathfrak{m}\oplus R/\mathfrak{m}$, which is a $2$-dimensional $R/\mathfrak{m}$ vector space, for which we know $\sigma(M, R)=|R/\mathfrak{m}|+1=\min_{\mathfrak{m}\in S_M}|R/\mathfrak{m}| +1$ by the previous paragraph, and results of Section~\ref{S3}. Now assume the result is true for any semi-simple $R$-module of length $k-1$ for some $k\geq 3$. Let $M$ be a semi-simple $R$-module of length $k$, for which $S_M\neq \emptyset$. Then $M\cong \bigoplus_{i=1}^{r}(R/\mathfrak{m}_i)^{k_i}\cong \bigoplus_{i=1}^{l}(R/\mathfrak{m}_i)^{k_i}\oplus (R/(\mathfrak{m}_{l+1}\cdots \mathfrak{m}_r))$, where $S_M=\lbrace \mathfrak{m}_1, \dots, \mathfrak{m}_l\rbrace$. Without loss of generality, let $\mathfrak{m}_1\in S_M$ be the one such that $|R/\mathfrak{m}_1|=\min_{\mathfrak{m}\in S_M}|R/\mathfrak{m}|$. Since $k_1\geq 2$, we know that $(R/\mathfrak{m}_1)^{k_1}$ has at least $|R/\mathfrak{m}_1|+1$ distinct maximal $R$-submodules (or $R/\mathfrak{m}_1$-subspaces). Let these be $\lbrace N_j\rbrace_{j\in\mathcal{K}}$, where $|\mathcal{K}|=\kappa= |R/\mathfrak{m}_1|+1$. Then we have at least $\kappa$-many length $k-1$ proper submodules of $M$, namely, $N_j\oplus\bigoplus_{i=2}^{l}(R/\mathfrak{m}_i)^{k_i}\oplus (R/(\mathfrak{m}_{l+1}\cdots \mathfrak{m}_r))$, $\forall j\in\mathcal{K}$.
	
	Assume $M=\bigcup_{j\in J} M_j$, where $|J|<\kappa$. Then there exists a length $k-1$ submodule of $M$, say $N$ of the form $N=N_{i_0}\oplus\bigoplus_{i=2}^{l}(R/\mathfrak{m}_i)^{k_i}\oplus (R/(\mathfrak{m}_{l+1}\cdots \mathfrak{m}_r))$, for some maximal $R$-submodule $N_{i_0}$ of $(R/\mathfrak{m}_1)^{k_1}$,  such that $N=\bigcup_{j\in J}(N\cap M_j)$ is a union of at most $|J|$ many proper submodules. Let $S_N$ be the collection of maximal ideals $\mathfrak{m}$ of $R$ for which $\dim_{R/\mathfrak{m}}(N/\mathfrak{m}N)\geq 2$. By the induction hypothesis, it follows that $|J|\geq \sigma(N, R)=\min_{\mathfrak{m}\in S_N}|R/\mathfrak{m}|+1$. Now $S_N=S_M$ if $\dim_{R/\mathfrak{m}_1}(N_{i_0})\geq2$, and $S_N=S_M\setminus\lbrace \mathfrak{m}_1\rbrace$ if $\dim_{R/\mathfrak{m}_1}(N_{i_0})=1$. In either case, $\min_{\mathfrak{m}\in S_N}|R/\mathfrak{m}|+1\geq \min_{\mathfrak{m}\in S_M}|R/\mathfrak{m}|+1$. Thus $|J|\geq \min_{\mathfrak{m}\in S_M}|R/\mathfrak{m}|+1=\kappa$, violating $|J|<\kappa$. Thus $\sigma(M, R)\geq \min_{\mathfrak{m}\in S_M}|R/\mathfrak{m}|+1 $. The reverse equality holds trivially, thereby proving the equality by induction.
\end{proof}

As an immediate consequence, one obtains the impossibility of a finite covering of a semi-simple module of \textit{any length} over a ring $R$ with infinite residue fields, by proper submodules. 

\begin{corollary}\label{C7}
	Let $M$ be \textit{any} semi-simple $R$-module over a ring $R$ with all residue fields infinite. Then $M$ is not the union of a finite number of proper submodules.
\end{corollary}

\begin{proof}
	Write $M=\bigoplus_{i\in I}N^{(i)}$, where $N^{(i)}$ are simple $R$-modules. Assume $M$ is a finite union of proper submodules $M_1, M_2, \dots, M_n$. Choose $x_j\in M\setminus M_j$. By the form of $M$, for each $j$, there exists a finite subset $F_j\subset I$, such that $x_j \in \bigoplus_{i\in F_j}N^{(i)}$. Consider the submodule $L$ of $M$ generated by $x_1, x_2, \dots, x_n$. Then $L\subseteq\bigoplus_{i \in \bigcup_j F_j}N^{(i)}$. Clearly $L$ is a semi-simple submodule of finite length. For each $j$, $M_j\cap L$ is a proper submodule of $L$ as $x_j\in L\setminus M_j$. Since $M=\bigcup_{i=1}^{n}M_i$, it implies $L=\bigcup_{i=1}^{n}M_i\cap L$. But this contradicts Proposition~\ref{P12}.
\end{proof}

We now strengthen Proposition~\ref{P12} to one of the main results of this paper-- see Theorem~\ref{introT2}, which we restate here for the reader's convenience.

\begin{theorem}\label{T5}
	Let $M$ be an $R$-module with small Jacobson radical (special case: $M$ is finitely generated) and finite dual Goldie dimension over a ring $R$. Let $S_M$ denote the set of maximal ideals $\mathfrak{m}$ of $R$ such that $\dim_{R/\mathfrak{m}}(M/\mathfrak{m}M)\geq 2$. If $S_M=\emptyset$ and $M$ is finitely generated, then $M$ is cyclic, and hence cannot be covered by proper submodules. Else if $S_M\neq\emptyset$, then $\sigma(M, R)=\min_{\mathfrak{m}\in S_M}|R/\mathfrak{m}|+1$.
\end{theorem}

Before proving the theorem, we first provide examples with finitely generated modules of finite dual Goldie dimension, to illustrate how Theorem~\ref{T5} generalizes the result for some previously known classes of modules. We begin by recalling the definition of radical of a module:

\begin{definition}[Radical]\label{D6}
	A submodule $N$ of an $R$-module $M$ (over a ring $R$) is called maximal if the quotient $M/N$ is a simple module. The \textit{radical} of the module $M$ is the intersection of all maximal submodules of $M$ and is denoted by $\jac(M)$.
\end{definition}

\begin{enumerate}
	\item Finitely generated modules over quasi-local rings have finite dual Goldie dimension.
	
	 To see this, first note that by Theorem $3.1.10 (1)$ of \cite{lomp}, if $M=\bigoplus_{i=1}^{n}M^{(i)}$, then $\hdim(M)=\hdim(M^{(1)})+\dots + \hdim(M^{(n)})$. The simplest examples of finitely generated modules for a given commutative ring $R$ are the free modules of finite rank. Thus, if $M=\bigoplus_{i}^{n}R$, then $\hdim(M)=n\cdot\hdim(R)$. By Corollary $3.3.5$ of \cite{lomp}, the commutative rings of finite dual Goldie dimension are precisely the quasi-local rings. Thus, if $R$ is a quasi-local ring, then any free $R$-module of finite rank is a module with finite dual Goldie dimension as well. By Theorem $3.1.10 (6)$ of \cite{lomp} and the exact sequence $0\rightarrow N\rightarrow R^n\rightarrow R^n/N\rightarrow 0$, it follows that $\hdim(R^n/N)\leq \hdim(R^n)$, from which one can conclude that any finitely generated module over a quasi-local commutative ring has finite dual Goldie dimension. 
	
	Conversely, the class of commutative rings, over which all finitely generated modules have finite dual Goldie dimension are precisely the class of commutative quasi-local rings, as one can consider the ring itself as a cyclic module over itself.\smallskip
	
	\item Any finite length module over a commutative ring has finite dual Goldie dimension. To see this, we need the following basic facts on finite length and semi-simple modules:

	\begin{proposition}\label{P9}
    Suppose $R$ is a unital commutative ring, and $M$ an $R$-module.
    \begin{enumerate}
        \item $M$ is of finite length if and only if it is both a Noetherian and an Artinian module.
        \item If $M$ is a semi-simple $R$-module then $\len(M)<\infty\iff$ $M$ is finitely generated.
        \item $M$ is finitely generated and semi-simple if and only if it is Artinian with radical zero.
    \end{enumerate}
\end{proposition}

  The proofs can be found in \cite[\S 1, 4, 9]{bourbaki}. This proposition, along with Theorem~\ref{T3} below (which is a crucial part in the proof of Theorem~\ref{T5}), proves the claim.\smallskip
	
	\item Every Artinian module has finite dual Goldie dimension, as shown in \cite{lomp}. Thus, Theorem~\ref{T5} also holds for any Artinian module with small radical over a commutative ring.\smallskip
	
	\item Referring to section $3.5.15$ of \cite{lomp}, one can get a complete classification of Abelian groups (i.e $\mathbb{Z}$-modules) of finite dual Goldie dimension as follows:
	\begin{enumerate}
		\item A non-zero torsion free Abelian group $A$ has $\hdim(A)=\infty$.
		\item $A$ is hollow, i.e., has dual Goldie dimension $1$ if and only if $A\cong \mathbb{Z}_{p^k}$ for some prime $p$ and $k\in\mathbb{N}\cup\lbrace\infty\rbrace$.
		\item $A$ has finite dual Goldie dimension if and only if it is a finite direct sum of hollow Abelian groups.
	\end{enumerate}
	
	 Since the Prüfer $p$-group $\mathbb{Z}_{p^\infty}$ is not finitely generated, it follows that Theorem~\ref{T5} holds for any finitely generated torsion Abelian group.\smallskip
	
	\item The previous example generalizes to a certain class of finitely generated modules over Dedekind domains. Hollow modules over Dedekind domains are characterized as follows:
	
	\begin{proposition}[Corollary $2.4$ of \cite{rangaswamy}]\label{46}
		Let $R$ be a Dedekind domain. Then an $R$-module $A$ is hollow if and only if: $(i)$ $A$ is an $R$-submodule of $K$ or $K/R$ where $K$ is the field of fractions of $R$, if $R$ is a discrete valuation ring; or $(ii)$ $A$ is a submodule of the Prüfer $P$-module $R(P^{\infty})$, for some non-zero prime ideal $P$, if $R$ is otherwise.
	\end{proposition}
	
	See \cite[Chapter 5]{insights} for a description of the submodules of the Prüfer $P$-module, which are isomorphic to $R/P^n$ for some $n\in\mathbb{N}$. It is well-known that any $R$-submodule of $K$ for a discrete valuation ring $R$ is of the form $R\cdot u^n$ for an integer $n$ and a fixed uniformizer $u$, and thus $R$-submodules of $K$ and $K/R$ include cyclic torsion $R$-modules.
	
	 Since the dual Goldie dimension is additive over direct sums, and hollow modules are modules with dual Goldie dimension $1$, Theorem~\ref{T5} holds for finite direct sums of hollow modules over Dedekind domains. A special case of this is Corollary $3.3$ of \cite{khare2020} since, besides finitely generated torsion modules over Dedekind domains, we show that Theorem~\ref{T5} also holds for modules of the form $\bigoplus_{i\in I} C^{(i)}$ where $C^{(i)}$ are (cyclic, but not necessarily torsion) $R$-submodules of $K$ or $K/R$ and $I$ is a finite set, and $R$ is a discrete valuation ring.\smallskip
	
	\item Theorem~\ref{T5} also holds for quasi-projective modules of finite dual Goldie dimension. Examples of these include finite direct sums of hollow quasi-projective modules. The author of \cite{rangaswamy} classifies hollow quasi-projective modules as those whose endomorphism rings are local. In particular, a projective module is hollow if and only if it is a local module (i.e. has a unique maximal submodule).
\end{enumerate}

The proof of Theorem~\ref{T5} requires the following basic characterization of the radical:

\begin{lemma}[Exercise $15.5$ in \cite{anderson}]\label{L5}
The Jacobson radical $\jac(M)$ of a module $M$ over a commutative ring $R$ is given by $\jac(M)=\bigcap_{\mathfrak{m}\in \mspec(R)}\mathfrak{m}M$.
\end{lemma}

 We also require the following result on modules with small Jacobson radical.
 
 \begin{theorem}[Theorem $1.13$ of \cite{dgdii}]\label{T3}
	Let $M$ be any module whose Jacobson radical $\jac(M)$ is small in $M$ and let $\hdim(M)$ denote the dual Goldie dimension of $M$. Then $\hdim(M)<\infty$ if and only if $M/\jac(M)$ is semi-simple Artinian, i.e., a finite length semi-simple module. Moreover, in this case $\hdim(M)$ equals the length of $M/\jac(M)$.  
\end{theorem}

With the above preliminaries, we now show:

\begin{proof}[Proof of Theorem~\ref{T5}]
	We will first show that $M$ is cyclic if $S_M=\emptyset$. By Lemma~\ref{L5}, the radical $\jac(M)\subset \mathfrak{m}M$ for all $\mathfrak{m}\in \mspec(R)$. Hence $\mathfrak{m}(M/\jac(M))= (\mathfrak{m}M+\jac(M))/(\jac(M))= \mathfrak{m}M/\jac(M)$. Thus, $M'/\mathfrak{m}M'\cong M/\mathfrak{m}M$ where $M'=M/\jac(M)$. Since $M$ has finite dual Goldie dimension, it follows from Theorem~\ref{T3} that $M'$ is a finite length semi-simple module. Hence, $M'/\mathfrak{m}M'\cong R/\mathfrak{m}$ for all $\mathfrak{m}\in \mspec(R)$, such that $\mathfrak{m}M'\subsetneq M'$, implies $M'$ is cyclic by Proposition~\ref{P12}. This implies $M$ is cyclic as well. To see this, assume $M$ is not cyclic. Then $M$ is a union of its maximal proper submodules, say $M=\bigcup_{i\in I} M_i$. Then $M'=M/\jac(M)=\bigcup_{i\in I} M_i/\jac(M)$. But since $\jac(M)\subset M_i$ for all $i\in I$, it follows that $M_i/\jac(M)$ is a proper submodule of $M/\jac(M)$. This would imply that $M'$, a cyclic module, is a union of its proper submodules, which is false. Thus $M$ is cyclic.
	
	Since $M$ has a small Jacobson radical $\jac(M)$, it follows that  if $M=\bigcup_{i\in I}M_i$ is a union of $|I|$-many proper submodules, then $M/\jac(M)=\bigcup_{i\in I}(M_i+\jac(M))/\jac(M)$ is also a union of at most $|I|$-many proper submodules. Conversely, if  $M/\jac(M)$ is a union of $|I'|$-many proper submodules, then its covering lifts to a covering of $M$ by at most $|I'|$-many proper submodules along the natural quotient map $q:M\rightarrow M/\jac(M)$. Hence, it follows that $\sigma(M, R)=\sigma(M/\jac(M), R)$. Since $M/\mathfrak{m}M\cong M'/\mathfrak{m}M'$ (where $M'=M/\jac(M)$), we see that $\lbrace \mathfrak{m}\in \mspec(R):$ $\dim_{R/\mathfrak{m}}(M/\mathfrak{m}M)\geq 2\rbrace=S_M= \lbrace \mathfrak{m}\in \mspec(R):$ $\dim_{R/\mathfrak{m}}(M'/\mathfrak{m}M')\geq 2\rbrace$. Hence, $\sigma(M, R)=\sigma(M/\jac(M), R)= \min_{\mathfrak{m}\in S_M}|R/\mathfrak{m}|+1$, where the second equality follows from Proposition~\ref{P12}. 
\end{proof}

Given Theorem \ref{T5} and the subsequent examples, we see that several previously studied families of modules -- for which the equality $\sigma(M,R) = \min_{\mathfrak{m}\in S_M}|R/\mathfrak{m}|+1 $ holds -- are indeed finitely generated with finite dual Goldie dimension. However, it is easily seen that this does not classify all modules for which $\sigma(M, R)=\min_{\mathfrak{m}\in S_M}|R/\mathfrak{m}|+1$. For example, \cite{rao} proves that this equality also holds for any Abelian group admitting a finite covering. Clearly, not all Abelian groups admitting a finite covering have finite dual Goldie dimension. For example, any finitely generated Abelian group $G$ with rank at least $2$ admits a finite covering, but $G$ does not have finite dual Goldie dimension, by the classification we saw above. Thus, the classification question (i.e. Question~\ref{Q2}) remains open. In the next section, we try to understand this problem topologically, akin to our analysis for vector spaces, and find topological criteria for the equality in Question~\ref{Q2} to hold.

\section{A topological perspective for modules}\label{S5}

All the notations of Section~\ref{S4} will be carried over in this section. In Section~\ref{S3.4} we saw a topological perspective of the covering problem for vector spaces. In this section, we will consider a similar topological interpretation of the corresponding problem for certain classes of modules. We will be interested in the classes of modules $M$ discussed in the previous section, for which one has the equality $\sigma(M, R)=\min_{\mathfrak{m}\in S_M}|R/\mathfrak{m}|+1$. This equality, along with the results of vector spaces, motivates one to appeal to the naturally available $R/\mathfrak{m}$-vector spaces $M/\mathfrak{m}M$ for the desired topology. Like before, $\sigma(M, R)$ will denote the minimum (cardinal) number of proper submodules of $M$, whose union covers the whole module $M$. Analogous to the case of vector spaces over infinite fields, we will define a topology on certain classes of finitely generated modules $M$, which we shall call the \textit{induced Zariski topology} on $M$, for reasons that will be evident later.

\begin{definition}
Given a finitely generated $R$-module $M$ over a ring $R$, equipped with the induced Zariski topology (see the next paragraph and Definition~\ref{izt}), the minimum (cardinal) number of closed subsets of $M$ whose union covers the whole space $M$ will be called the \textit{Zariski covering number} of $M$ and denoted by $\sigma_{\tau}(M, R)$.
\end{definition}

\subsection{The topology for finitely generated modules with finite dual Goldie dimension}

We first define the induced Zariski topology and compute the Zariski covering number in a special case. Let $(R, \mathfrak{m})$ be a local ring with an infinite residue field and let $M$ be any finitely generated $R$-module, admitting a cover by proper submodules. Then $M/\mathfrak{m}M$ is a finite-dimensional $R/\mathfrak{m}$-vector space. Equip $M/\mathfrak{m}M$ with the canonical induced Zariski topology, by considering it as an $R/\mathfrak{m}$-vector space. Now equip $M$ with the coarsest topology to make the natural $R$-linear map $\pi: M\rightarrow M/\mathfrak{m}M$ continuous: in particular, define the open sets of $M$ to be $\pi^{-1}(\mathcal{O})$, where $\mathcal{O}\subseteq M/\mathfrak{m}M$ are the open subsets of the induced Zariski topology on $M/\mathfrak{m}M$ as an $R/\mathfrak{m}$-vector space. This is the definition of the induced Zariski topology on a finitely generated module $M$ over a local ring $(R, \mathfrak{m})$. With this topology on $M$, we have the following result.

\begin{proposition}\label{P13}
	Let $(R, \mathfrak{m})$ be a local ring with infinite residue field $R/\mathfrak{m}$, and $M$ be a finitely generated $R$-module admitting a covering by proper submodules. Consider $M$ as a topological space with the induced Zariski topology. Then $\sigma_{\tau}(M, R)=|R/\mathfrak{m}|+1$. In particular, $\sigma(M, R)=|R/\mathfrak{m}|+1$.
\end{proposition}

\begin{proof}
	Since $M$ admits a covering by proper submodules, it follows from the analysis of Section~\ref{S4} that $\dim_{R/\mathfrak{m}}(M/\mathfrak{m}M)\geq 2$. The proper closed sets of $M$ are precisely $\pi^{-1}(\mathcal{C})$ for a proper closed subset $\mathcal{C}$ (with respect to the induced Zariski topology) of $M/\mathfrak{m}M$. Let $\lbrace \pi^{-1}(\mathcal{C}_i)\rbrace_{i\in I} $ be a collection of proper closed subsets of $M$ that cover $M$. Under the surjective map $\pi$, we see that $M/\mathfrak{m}M=\pi(M)=\pi(\bigcup_{i\in I} \pi^{-1}(\mathcal{C}_i))=\bigcup_{i\in I} \pi(\pi^{-1}(\mathcal{C}_i))=\bigcup_{i\in I}\mathcal{C}_i$. Thus, the collection of proper closed sets $\lbrace \mathcal{C}_i\rbrace_{i\in I} $ forms a cover for the $R/\mathfrak{m}$-vector space $M/\mathfrak{m}M$. From Theorem~\ref{T2}, it follows that then $\mathfrak{I}=|I|\geq |R/\mathfrak{m}|+1$. In particular, $M/\mathfrak{m}M$ has a minimal cover of size $|R/\mathfrak{m}|+1$ by results of Section~\ref{S3}, and lifting this to $M$, we see that $\sigma_{\tau}(M, R)=|R/\mathfrak{m}|+1=\sigma(M, R)$.
\end{proof}

We will now generalize the above topology to the class of finitely generated modules $M$ with finite dual Goldie dimension, admitting a covering by proper submodules over a commutative ring $R$, \textbf{whose residue fields $R/\mathfrak{m}$ have infinite cardinalities for all $\mathfrak{m}\in S_M$}. We saw earlier that such a module $M$ admits a covering by proper submodules if and only if $S_M$ is non-empty, where $S_M$ is the set of maximal ideals $\mathfrak{m}$ of $R$ such that $\dim_{R/\mathfrak{m}}(M/\mathfrak{m}M)\geq 2$.

\begin{definition}[Topologizing $M/\mathfrak{m}M$]\label{factordef}
    Let $M$ be a finitely generated $R$-module such that for all $\mathfrak{m}\in S_M$, $R/\mathfrak{m}$ is infinite. We define the \textit{factor spaces of $M$} to be the quotient modules $M/\mathfrak{m}M$ (for all $\mathfrak{m}\in\mspec(R)$), topologized as follows:
    \begin{enumerate}
	\item For $\mathfrak{m}\in S_M\subseteq \mspec(R)$, we view the $R/\mathfrak{m}$-vector space $M/\mathfrak{m}M$ of dimension at least $2$, as a topological space with the induced Zariski topology as described in Section~\ref{S3.4}. 
	\item For $\mathfrak{m}\in \mspec(R)\setminus S_M$, we view the (at most $1$-dimensional) $R/\mathfrak{m}$-vector space $M/\mathfrak{m}M$ as a topological space, with the only closed sets being $\emptyset, \lbrace 0\rbrace$ and $M/\mathfrak{m}M$. We will call this the \textit{discrete subspace topology}. 
\end{enumerate}
\end{definition}

It is easy to see that all the non-empty open sets in the factor spaces of $M$ are dense in the corresponding factor space. For each $\mathfrak{m}\in \mspec(R)$, we have an associated $R/\mathfrak{m}$ vector space $M/\mathfrak{m}M$ along with the quotient map $q_{\mathfrak{m}}: M\rightarrow M/\mathfrak{m}M$. We patch these vector spaces up to form the $R$-module $\prod_{\mathfrak{m}\in \mspec(R)}M/\mathfrak{m}M$, with the product topology, where the terms $M/\mathfrak{m}M$ in the product are viewed as factor spaces as defined above. For our convenience, we will always ignore the trivial components (i.e. $\mathfrak{m}\in\mspec(R)$ such that $\mathfrak{m}M=M$) in $\prod_{\mathfrak{m}\in \mspec(R)}M/\mathfrak{m}M$. We have a natural associated map $q:M\rightarrow \prod_{\mathfrak{m}\in \mspec(R)}M/\mathfrak{m}M$, given by $q(x)=(q_{\mathfrak{m}}(x))_{\mathfrak{m}\in \mspec(R)}$. The image of $q$ is isomorphic to $M/\jac(M)$ by Lemma~\ref{L5}. 

Now we focus on finitely generated modules $M$ with finite dual Goldie dimension, i.e., such that $M/\jac(M)$ is a finite length semi-simple module. In other words, $M/\jac(M)\cong \bigoplus_{i=1}^{n}(R/\mathfrak{m}_i)$ for some $\mathfrak{m}_i\in \mspec(R)$, where the isomorphism is given by an $R$-module homomorphism $\Phi$. Without loss of generality, let the distinct maximal ideals among $\lbrace\mathfrak{m}_i\rbrace_{i=1}^{n}$ be $\mathfrak{m}_1, \dots, \mathfrak{m}_r$ for some $r\leq n$. Then $\jac(M)\subseteq \mathfrak{m}M$ for all $\mathfrak{m}\in \mspec(R)$ by Lemma~\ref{L5}, so: $$\mathfrak{m}M/\jac(M)= \mathfrak{m}(M/\jac(M))\cong \mathfrak{m}(\bigoplus_{i=1}^{n}(R/\mathfrak{m}_i))= \bigoplus_{i=1}^{n}\mathfrak{m}(R/\mathfrak{m}_i), \qquad \forall \mathfrak{m}\in\mspec(R),$$ where the isomorphism in the middle is the restriction of $\Phi$ to the submodule $\mathfrak{m}M/\jac(M)$ of $M/\jac(M)$. Now since $\mathfrak{m}(R/\mathfrak{m}_i)$ equals $0$ if $\mathfrak{m}=\mathfrak{m}_i$ and $R/\mathfrak{m}_i$ otherwise, it follows that:
$$M/\mathfrak{m}M\cong \frac{M/\jac(M)}{\mathfrak{m}M/\jac(M)}\cong \frac{\bigoplus_{i=1}^{n}(R/\mathfrak{m}_i)}{\bigoplus_{i=1}^{n}\mathfrak{m}(R/\mathfrak{m}_i)}\cong\bigoplus_{i=1}^{n}\frac{(R/\mathfrak{m}_i)}{\mathfrak{m}(R/\mathfrak{m}_i)} \cong (R/\mathfrak{m})^{k_\mathfrak{m}}$$
 where $k_\mathfrak{m}$ is the number of copies of $R/\mathfrak{m}$ in the semi-simple decomposition of $M/\jac(M)$. Thus, $M/\mathfrak{m}M\neq 0$ only for $\mathfrak{m}=\mathfrak{m}_i$, where $i=1, \dots, r$. Hence, for a finitely generated module $M$ with finite dual Goldie dimension, the product $\prod_{\mathfrak{m}\in \mspec(R)}M/\mathfrak{m}M$ is finite, and equals $\prod_{i=1}^{r}M/\mathfrak{m}_iM$, where $\mathfrak{m}_i$ for $i=1, \dots, r$ are the distinct maximal ideals occurring in the semi-simple decomposition of $M/\jac(M)$. Conversely, if $\mspec(R)$ is finite, then $M/\jac(M) \cong \prod_{\mathfrak{m} \in \mspec(R)} M/\mathfrak{m} M = \bigoplus_{\mathfrak{m} \in \mspec(R)} M/\mathfrak{m} M$
by the Chinese remainder theorem. It follows that $M/\jac(M)$ is isomorphic to a finite direct sum of simple modules. \textbf{Thus, $\prod_{\mathfrak{m}\in \mspec(R)}M/\mathfrak{m}M$ is a finite product if and only if $M$ has finite dual Goldie dimension}.

We now add the following hypothesis: $M$ is a finitely generated $R$-module with finite dual Goldie dimension, such that \textbf{each simple factor, with multiplicity at least $2$, in the semi-simple decomposition of $M/\jac(M)$ is of infinite cardinality}. In terms of the above setup, the added hypothesis requires $R/\mathfrak{m}_i$ to be infinite fields for those $i$, such that $k_{\mathfrak{m}_i}\geq 2$. 

\begin{definition}[Induced Zariski topology on $M$]\label{izt}
    With setup as above, topologize each $M/\mathfrak{m}_i M$ according to Definition~\ref{factordef}, and equip $\prod_{i=1}^{r}M/\mathfrak{m}_i M$ with the product topology. The isomorphism $M/\jac(M)\cong \prod_{i=1}^{r}M/\mathfrak{m}_i M$ induces a topology on $M/\jac(M)$, which would make the first isomorphism theorem map into a homeomorphism. Finally, we equip $M$ with the coarsest topology that makes the quotient map $q: M\rightarrow M/\jac(M)$ continuous. This is the definition of the \textit{induced Zariski topology} on a finitely generated module $M$ with finite dual Goldie dimension.
\end{definition}

\textbf{Explicit description of the induced Zariski topology on $M$:} Consider the map $q: M\rightarrow \prod_{i=1}^{r}M/\mathfrak{m}_i M $ with the natural quotient maps $q_i: M\rightarrow M/\mathfrak{m}_iM$ being the component maps of $q$. Let $\prod_{i=1}^{r}\mathcal{O}_i$ be a basic open subset of the product space, equipped with the product topology as described earlier. Then the induced Zariski topology on $M$ can be described as the one whose basic open sets are $q^{-1}(\prod_{i=1}^{r}\mathcal{O}_i)=\cap_{i=1}^{r}q_i^{-1}(\mathcal{O}_i)$, where $\mathcal{O}_i\subseteq M/\mathfrak{m}_iM$ are open, $\forall i=1, ,2\dots, r$.  

\begin{remark}\label{R10}
	As we saw in Section~\ref{S3.4}, the topology on $M/\mathfrak{m}M$ (as induced by an isomorphism with an affine space $(R/\mathfrak{m})^k$ for some $k>1$) is independent of the basis chosen for the isomorphism, thereby making the topology canonical. Hence, the topology on $M$ is canonical as well.
\end{remark}

We now come to one of the main results of this paper.

\begin{theorem}\label{T6}
	Let $M$ be a finitely generated non-cyclic $R$-module with finite dual Goldie dimension, such that the residue fields $R/\mathfrak{m}$ are infinite for all maximal ideals $\mathfrak{m}\in S_M$. Equip $M$ with the induced Zariski topology as defined above. Then $\sigma_{\tau}(M, R)=\min_{\mathfrak{m}\in S_M}|R/\mathfrak{m}|+1$. Consequently, $(a)$ the induced Zariski topology on $M$ makes it into a $\kappa_M$-maximal Baire space, for $\kappa_M=\min_{\mathfrak{m}\in S_M}|R/\mathfrak{m}|+1 $; and $(b)$ the covering number of $M$ is $\sigma(M, R)=\kappa_M$.
\end{theorem}

\begin{proof}
	We will prove that $M$ with the induced Zariski topology is a $\kappa_M$-Baire space for $\kappa_M= \min_{\mathfrak{m}\in S_M}|R/\mathfrak{m}|+1$. Let $M/\jac(M)\cong \bigoplus_{i=1}^{n} R/\mathfrak{m}_i$ and without loss of generality let the distinct maximal ideals among the $\mathfrak{m}_i$ be $\mathfrak{m}_1, \dots, \mathfrak{m}_r$. Then $M/\jac(M)\cong \prod_{i=1}^{r}M/\mathfrak{m}_iM$. We have the natural surjective map $q:M\rightarrow \prod_{i=1}^{r}M/\mathfrak{m}_iM$ inducing the topology on $M$ from the product topology of the product. Let $\lbrace \mathcal{U}_i\rbrace_{i\in I}$ be a family of basic open sets in $M$ where $\mathcal{U}_i=q^{-1}(\prod_{j=1}^{r}\mathcal{O}_{ji})$. Then
	$$\bigcap_{i\in I}\mathcal{U}_i=\bigcap_{i\in I}q^{-1}(\prod_{j=1}^{r}\mathcal{O}_{ji})=q^{-1}(\bigcap_{i\in I}(\prod_{j=1}^{r}\mathcal{O}_{ji}))=q^{-1}(\prod_{j=1}^{n}(\bigcap_{i\in I}\mathcal{O}_{ji})).$$
	Now note, $\prod_{j=1}^{r}(\bigcap_{i\in I}\mathcal{O}_{ji})\neq \emptyset \iff \bigcap_{i\in I}\mathcal{O}_{ji}\neq \emptyset$ in $M/\mathfrak{m}_jM$, $\forall j=1, 2, \dots, r$. For any index $p$ such that $M/\mathfrak{m}_pM\cong R/\mathfrak{m}_p$, the discrete subspace topology on $M/\mathfrak{m}_pM$ ensures that arbitrary intersections of open sets are non-empty. So,
	$\prod_{j=1}^{r}(\bigcap_{i\in I}\mathcal{O}_{ji})\neq \emptyset \iff \bigcap_{i\in I}\mathcal{O}_{ji}\neq \emptyset$ in $M/\mathfrak{m}_jM$, $\forall \mathfrak{m}_j\in S_M$, for $j=1, \dots, r$. 
	
	Now for $\mathfrak{m}_j\in S_M$, $\bigcap_{i\in I}\mathcal{O}_{ji}\neq \emptyset$ for any family $\lbrace \mathcal{O}_{ji}\rbrace_{i\in I}$ of basic open sets in $M/\mathfrak{m}_jM$ of size $|I|$ $\iff |I|< |R/\mathfrak{m}_j|+1$, by results of section $2.4$. Thus, $\prod_{j=1}^{r}(\bigcap_{i\in I}\mathcal{O}_{ji})\neq \emptyset \iff |I|<\min_{\mathfrak{m}\in S_M}|R/\mathfrak{m}|+1$. Since $q$ is surjective, it follows that $\bigcap_{i\in I}\mathcal{U}_i=\bigcap_{i\in I}q^{-1}(\prod_{j=1}^{r}\mathcal{O}_{ji})=q^{-1}(\bigcap_{i\in I}(\prod_{j=1}^{r}\mathcal{O}_{ji}))\neq \emptyset$, for any family $\lbrace \mathcal{U}_i\rbrace_{i\in I}$ of basic open sets in $M$ of size $|I|$ if and only if $|I|<\min_{\mathfrak{m}\in S_M}|R/\mathfrak{m}|+1$. Thus, taking complements, it follows that $\sigma_{\tau}(M, R)=\min_{\mathfrak{m}\in S_M}|R/\mathfrak{m}|+1$.
	
	The induced topology on $M$ makes $q$ into an open continuous map. One can check that the inverse image of any dense set under any open map is dense. We also know that the product of dense sets is dense in the product space and that every open set in each of the factor spaces is dense in the corresponding factor space. These together imply that every open subset of the induced Zariski topology on $M$ is dense in the space $M$. Now $\sigma_{\tau}(M, R)=\kappa_M$ along with the density of the open sets in $M$ under the induced Zariski topology are equivalent to $M$ being a $\kappa_M$-maximal Baire space, for $\kappa_M=\min_{\mathfrak{m}\in S_M}|R/\mathfrak{m}|+1 $.
\end{proof}

\subsection{The topology for general finitely generated modules}

All rings considered are unital and commutative, as before. \textbf{Throughout this section, we will assume that given a finitely generated $R$-module $M$, the ring $R$ is such that the quotient field $R/\mathfrak{m}$ is infinite for all $\mathfrak{m}\in S_M=\lbrace \mathfrak{m}\in \mspec(R):$ $\dim_{R/\mathfrak{m}}(M/\mathfrak{m}M)\geq2\rbrace$}.

Recall from Definition~\ref{factordef} how the factor space $M/\mathfrak{m}M$ is topologized, for $\mathfrak{m}\in \mspec(R)$. It is immediate that a factor space with the discrete subspace topology is compact. For factor spaces with the  induced Zariski topology, one sees that $M/\mathfrak{m}M$ is a Noetherian and hence compact topological space. Thus, using the Tychonoff theorem, we have:

\begin{lemma}\label{L6}
	The space $\prod_{\mathfrak{m}\in \mspec(R)}M/\mathfrak{m}M$, with the product topology derived from the factor spaces of $M$, is compact. In fact, if $S_M$ equals $\mspec(R)$, then $\prod_{\mathfrak{m}\in \mspec(R)}M/\mathfrak{m}M$ is a compact $T1$ space.
\end{lemma}

In the case of finite dual Goldie dimension modules, the topology on $M$ is induced by the natural map to a finite product of the factor spaces of $M$. In the general case, the product $\prod_{\mathfrak{m}\in \mspec(R)}M/\mathfrak{m}M$ may be infinite. Nevertheless, we now show that one can always interpret the topological space $\prod_{\mathfrak{m}\in \mspec(R)}M/\mathfrak{m}M$ as obtained by ``gluing" together all possible copies of products (or direct sums) of finitely many spaces $M/\mathfrak{m}M$, where $\mathfrak{m}\in \mspec(R)$. 

For any finite subset $\sigma=\lbrace \mathfrak{m}_1, \dots, \mathfrak{m}_k\rbrace$ of $\mspec(R)$, we have the natural map $q_{\sigma}:M\rightarrow \prod_{\mathfrak{m}\in \sigma}M/\mathfrak{m}M$. These maps are easier to understand, since these are surjective, unlike the map $q: M\rightarrow \prod_{\mathfrak{m}\in \mspec(R)}M/\mathfrak{m}M$, which may not be surjective. 

Let $\mathcal{F}=\lbrace \sigma\subseteq \mspec(R):$ $|\sigma|<\infty\rbrace$ be a directed poset with $\sigma_1\preccurlyeq \sigma_2$ if $\sigma_1\subseteq \sigma_2$, for all $\sigma_1, \sigma_2\in \mathcal{F}$. For each $\sigma\in \mathcal{F}$, consider $M_\sigma=\prod_{\mathfrak{m}\in\sigma}M/\mathfrak{m}M$ as a topological space, with the product topology derived from the factor spaces occurring in the product. Now consider a family of continuous maps $p_{ij}: M_{\sigma_j}\rightarrow M_{\sigma_i}$ for $\sigma_i\preccurlyeq \sigma_j$ as follows:
\begin{enumerate}
	\item $p_{ii}$ is the identity map on $M_{\sigma_i}$, for all $\sigma_i\in \mathcal{F}$.
	\item $p_{ij}: \prod_{\mathfrak{m}\in \sigma_j}M/\mathfrak{m}M\rightarrow \prod_{\mathfrak{m}\in \sigma_i}M/\mathfrak{m}M$ is the projection map for $\sigma_i\preccurlyeq \sigma_j$.
\end{enumerate}

Clearly, this defines an inverse system $\Sigma_1=((M_{\sigma})_{\sigma\in\mathcal{F}}, (p_{ij})_{\sigma_i\preccurlyeq\sigma_j\in\mathcal{F}})$ of topological spaces and surjective $R$-linear continuous maps, with inverse limit $\varprojlim_{\sigma\in\mathcal{F}}M_{\sigma}$.  The inverse limit comes along with natural projection maps $\pi_\sigma: \varprojlim_{\sigma\in\mathcal{F}}M_{\sigma}\rightarrow M_{\sigma}$ for all $\sigma\in\mathcal{F}$.

It is immediate from the universal property of inverse limits, and the existence of natural projection maps $f_{\sigma}:\prod_{\mathfrak{m}\in \mspec(R)}M/\mathfrak{m}M\rightarrow M_\sigma $, that one always has a unique continuous map $\tilde{f}:\prod_{\mathfrak{m}\in \mspec(R)}M/\mathfrak{m}M\rightarrow \varprojlim_{\sigma\in\mathcal{F}}M_{\sigma}$, such that $\pi_\sigma\circ\tilde{f}=f_\sigma$. We claim that this map is bijective.

\begin{proposition}\label{P14}
	There exists a unique $R$-linear homeomorphism between $\prod_{\mathfrak{m}\in \mspec(R)}M/\mathfrak{m}M$ and the inverse limit $\varprojlim_{\sigma\in\mathcal{F}}M_{\sigma}$ of the inverse system $((M_{\sigma})_{\sigma\in\mathcal{F}}, (p_{ij})_{\sigma_i\preccurlyeq\sigma_j\in\mathcal{F}})$.
\end{proposition}

\begin{proof}
Note that the inverse limit has an explicit description as:
$$\varprojlim_{\sigma\in\mathcal{F}}M_{\sigma}=\lbrace (x_\sigma)_{\sigma\in\mathcal{F}}\in \prod_{\sigma\in\mathcal{F}}M_{\sigma} \ \text{such that } p_{ij}\circ\pi_{\sigma_j}((x_\sigma)_{\sigma\in\mathcal{F}})=\pi_{\sigma_i}((x_\sigma)_{\sigma\in\mathcal{F}}) \ \forall \sigma_i\preccurlyeq\sigma_j\rbrace.$$ We have a unique continuous map $\tilde{f}:\prod_{\mathfrak{m}\in \mspec(R)}M/\mathfrak{m}M\rightarrow \varprojlim_{\sigma\in\mathcal{F}}M_{\sigma}$, such that $\pi_\sigma\circ\tilde{f}=f_\sigma$. Let $(x_\sigma)_{\sigma\in\mathcal{F}}\in \varprojlim_{\sigma\in\mathcal{F}}M_{\sigma}$ be an arbitrary element. Let $\mathfrak{m}$ denote the element $\lbrace\mathfrak{m}\rbrace\in\mathcal{F}$ for each $\mathfrak{m}\in\mspec(R)$ and consider the element $x=(x_\mathfrak{m})_{\mathfrak{m}\in \mspec(R)}\in \prod_{\mathfrak{m}\in \mspec(R)}M/\mathfrak{m}M$. Then $\tilde{f}(x)=(f_{\sigma}(x))_{\sigma\in\mathcal{F}}\in \varprojlim_{\sigma\in\mathcal{F}}M_{\sigma}$, where $f_{\sigma}(x)=(x_\mathfrak{m})_{\mathfrak{m}\in\sigma}$. 

Since $\mspec(R)$ can be realised as a subset of $\mathcal{F}$ by identifying $\mathfrak{m}\in \mspec(R)$ with the element $\lbrace\mathfrak{m}\rbrace\in\mathcal{F}$, there is a natural continuous projection map $g: \prod_{\sigma\in\mathcal{F}}M_{\sigma}\rightarrow \prod_{\mathfrak{m}\in \mspec(R)}M/\mathfrak{m}M$. Let $\tilde{g}$ be the restriction of $g$ to the inverse limit, which is a subspace of $\prod_{\sigma\in\mathcal{F}}M_{\sigma}$. Thus, we have a continuous map $\tilde{g}: \varprojlim_{\sigma\in\mathcal{F}}M_{\sigma}\rightarrow \prod_{\mathfrak{m}\in \mspec(R)}M/\mathfrak{m}M$, which maps an element $(x_\sigma)_{\sigma\in\mathcal{F}}\in \varprojlim_{\sigma\in\mathcal{F}}M_{\sigma}$ to $(x_\mathfrak{m})_{\mathfrak{m}\in \mspec(R)}\in \prod_{\mathfrak{m}\in \mspec(R)}M/\mathfrak{m}M$, where $x_{\mathfrak{m}}=x_{\lbrace\mathfrak{m}\rbrace}$. 

It can be easily checked that $\tilde{g}\circ\tilde{f}=\mathbf{1}_{\Pi}$, where $\mathbf{1}_{\Pi}$ is the identity map on $\prod_{\mathfrak{m}\in \mspec(R)}M/\mathfrak{m}M$. Similarly, one sees that $\tilde{f}\circ\tilde{g}=\mathbf{1}_{inv}$, where $\mathbf{1}_{inv}$ is the identity map on $\varprojlim_{\sigma\in\mathcal{F}}M_{\sigma}$.
\end{proof}

Now consider the directed poset $\mathcal{I}(R)=\lbrace$non-trivial proper ideals $I\subsetneq R\rbrace$, with partial order $\preccurlyeq$ where $I_1\preccurlyeq I_2$ if and only if $I_2\subseteq I_1$. For each ideal $I\in\mathcal{I}(R)$, consider the quotient $M/IM$, which is also an $R/I$ module. We wish to topologize this. Consider the collection $m(I)$ of maximal ideals of $R$ containing $I$. For each $\mathfrak{m}\in m(I)$, there exists a projection map $p_{\mathfrak{m}, I}: M/IM\rightarrow M/\mathfrak{m}M$. This results in a natural map $p_{I}=(p_{\mathfrak{m}, I})_{\mathfrak{m}\in m(I)}: M/IM\rightarrow \prod_{\mathfrak{m}\in m(I)}M/\mathfrak{m}M$. We already consider $M/\mathfrak{m}M$ as a factor space, with topology as determined by Definition~\ref{factordef}, for all $\mathfrak{m}\in \mspec(R)$. Now equip $\prod_{\mathfrak{m}\in m(I)}M/\mathfrak{m}M$ with the resulting product topology, and topologize $M/IM$ with the coarsest topology, which makes the map $p_I$ continuous.

Note that if $I\preccurlyeq J$, then $J\subseteq I$ and so $m(I)\subseteq m(J)$. Thus, there is a natural continuous projection map $\tilde{q}_{IJ}:\prod_{\mathfrak{m}\in m(J)}M/\mathfrak{m}M\rightarrow \prod_{\mathfrak{m}\in m(I)}M/\mathfrak{m}M$. We also have a natural quotient map $q_{IJ}: M/JM\rightarrow M/IM$ . One can check that the following diagram commutes:
$$\begin{tikzcd}
M/JM \arrow[d, "p_J"'] \arrow[r, "q_{IJ}", two heads] & M/IM \arrow[d, "p_I"] \\
\prod_{\mathfrak{m}\in m(J)}M/\mathfrak{m}M \arrow[r, "\tilde{q}_{IJ}"', two heads]                & \prod_{\mathfrak{m}\in m(I)}M/\mathfrak{m}M               
\end{tikzcd}$$

By the definition of the topology on $M/IM$, the open sets are $p_I^{-1}(U)$ for open subsets $U\subseteq \prod_{\mathfrak{m}\in m(I)}M/\mathfrak{m}M$. By commutativity of the above diagram, $q_{IJ}^{-1}(p_I^{-1}(U))=p_J^{-1}(\tilde{q}_{IJ}^{-1}(U))$. Now the continuity of $\tilde{q}_{IJ}$ and $p_J$
implies that of $q_{IJ}$. Also, for each $I\in \mathcal{I}(R)$, define $q_{II}$ to be the identity map on $M/IM$. This defines an inverse system $\Sigma_2=((M/IM)_{I\in\mathcal{I}(R)}, (q_{IJ})_{I\preccurlyeq J\in\mathcal{I}(R)})$ of topological spaces and surjective continuous  $R$-linear maps.

Recalling the inverse system $\Sigma_1=((M_{\sigma})_{\sigma\in\mathcal{F}}, (p_{ij})_{\sigma_i\preccurlyeq\sigma_j\in\mathcal{F}})$ studied in Proposition~\ref{P14}, we have a mapping $\lbrace\phi, \lbrace\psi_{\sigma}\rbrace_{\sigma\in\mathcal{F}}\rbrace$ of inverse systems from $\Sigma_1$ to $\Sigma_2$ as follows (see \cite[pp.~101]{engelking} for relevant definitions):

\begin{enumerate}
	\item $\phi: \mathcal{F}\rightarrow \mathcal{I}(R)$, defined by mapping $\sigma\in\mathcal{F}$ to $\bigcap_{\mathfrak{m}\in\sigma}\mathfrak{m}\in\mathcal{I}(R)$.
	\item The isomorphism map $\psi_{\sigma}:M_\sigma\rightarrow M/(\bigcap_{\mathfrak{m}\in\sigma}\mathfrak{m}M)$ for all $\sigma\in\mathcal{F}$.
\end{enumerate}

Clearly $\phi(\mathcal{F})$ is a cofinal subset of $\mathcal{I}(R)$, since for every ideal $I\in \mathcal{I}(R)$, there exists some maximal ideal $\mathfrak{m}$ containing it. Then by Proposition $2.5.10$ of \cite{engelking}, it follows that:

\begin{proposition}\label{P15}
	There is a mapping $\lbrace\phi, \lbrace\psi_{\sigma}\rbrace_{\sigma\in\mathcal{F}}\rbrace$ of inverse systems from $\Sigma_1=((M_{\sigma})_{\sigma\in\mathcal{F}}, (p_{ij})_{\sigma_i\preccurlyeq\sigma_j\in\mathcal{F}})$ to $\Sigma_2=((M/IM)_{I\in\mathcal{I}(R)}, (q_{IJ})_{I\preccurlyeq J\in\mathcal{I}(R)})$, which induces a homeomorphism between the inverse limit $\varprojlim_{\sigma\in\mathcal{F}}M_{\sigma}$ of $\Sigma_1$ and the inverse limit $\varprojlim_{I\in\mathcal{I}(R)}M/IM$ of $\Sigma_2$.
\end{proposition}

As a consequence of Propositions~\ref{P15} and \ref{P14}, we obtain:

\begin{corollary}\label{C13}
	The topological spaces $\prod_{\mathfrak{m}\in \mspec(R)}M/\mathfrak{m}M$, $\varprojlim_{\sigma\in\mathcal{F}}M_{\sigma}$, and $\varprojlim_{I\in\mathcal{I}(R)}M/IM$ are all homeomorphic via the natural maps.
\end{corollary}

\begin{remark}\label{R11}
\begin{enumerate}
    \item Corollary~\ref{C13} shows that the three ways of assembling ``simpler" topological spaces result in larger spaces that are all naturally homeomorphic to each other. By the universal property of inverse limits, there exist natural $R$-linear maps $q:M\rightarrow \prod_{\mathfrak{m}\in \mspec(R)}M/\mathfrak{m}M$, $\Psi: M\rightarrow \varprojlim_{\sigma\in\mathcal{F}}M_{\sigma}$ and $\Phi: M\rightarrow \varprojlim_{I\in\mathcal{I}(R)}M/IM$. Since these three spaces are naturally homeomorphic, it will follow that if the topology on $M$ makes one of the maps $q, \Psi$ or $\Phi$ continuous, then the other two shall be continuous too.\smallskip
    
    \item In each of the three spaces in Corollary~\ref{C13}, the ``ill-behaved" factors are the ones coming from the  factor spaces $M/\mathfrak{m}M$ with the discrete subspace topology, which is a coarse topology (not even $T0$). This artificially introduced coarseness ensures that the factors $M/\mathfrak{m}M$ with $\dim_{R/\mathfrak{m}}(M/\mathfrak{m}M)=1$ do not affect the quantity $\sigma_{\tau}(M, R)$. Thus, the only factor spaces which influence $\sigma_{\tau}(M, R)$ are the ones corresponding to the maximal ideals in $S_M=\lbrace \mathfrak{m}\in \mspec(R):$ $\dim_{R/\mathfrak{m}}(M/\mathfrak{m}M)\geq 2\rbrace$.
    
     Letting $\mathcal{F}_S=\lbrace \sigma\in\mathcal{F}:$ $\sigma\subseteq S_M\rbrace$, one can show that the inverse limit $\varprojlim_{\sigma\in\mathcal{F}_S}M_{\sigma}$ is homeomorphic to $\prod_{\mathfrak{m}\in S_M}M/\mathfrak{m}M$, thereby proving that $\varprojlim_{\sigma\in\mathcal{F}_S}M_{\sigma}$ is a compact $T1$ space. If the factors were finite, then the induced Zariski topology would be discrete on each factor, and one would instead get the profinite topology, which is a compact Hausdorff space. Thus, the topology we defined on the product of the well-behaved components, can be considered as a $T1$ generalization of the profinite topology. 
\end{enumerate}
\end{remark}

Now recall the natural map $q: M\rightarrow \prod_{\mathfrak{m}\in \mspec(R)}M/\mathfrak{m}M$. The image $q(M)$ is topologized by equipping it with the subspace topology of the product space $\prod_{\mathfrak{m}\in \mspec(R)}M/\mathfrak{m}M$. Since $q(M)\cong M/\jac(M)$, one can topologize $M$ by equipping it with the coarsest topology which makes the map $q$ continuous. Note, $q(M)$ is not an arbitrary subspace of $\prod_{\mathfrak{m}\in \mspec(R)}M/\mathfrak{m}M$:

\begin{proposition}\label{P16}
	The image of $M$ under the continuous map $q: M\rightarrow \prod_{\mathfrak{m}\in \mspec(R)}M/\mathfrak{m}M$, is dense in the product space.
\end{proposition}

This result shows that the space $\prod_{\mathfrak{m}\in \mspec(R)}M/\mathfrak{m}M$ provides a compactification of $M/\jac(M)$.

\begin{proof}
	Let $\prod_{\mathfrak{m}\in \mspec(R)}\mathcal{O}_{\mathfrak{m}}$ be a basic open set in $\prod_{\mathfrak{m}\in \mspec(R)}M/\mathfrak{m}M$-- so except finitely many indices $\mathfrak{m}_1, \dots, \mathfrak{m}_k$, the remaining $\mathcal{O}_{\mathfrak{m}}$ equal $M/\mathfrak{m}M$. By the Chinese remainder theorem, the natural quotient map $f:M\rightarrow \prod_{i=1}^{k}M/\mathfrak{m}_iM$ is surjective, so there exists $x\in M$ such that $f(x)\in \prod_{i=1}^{k}\mathcal{O}_{\mathfrak{m}_i}$. Now $q(x)\in \prod_{\mathfrak{m}\in \mspec(R)}\mathcal{O}_{\mathfrak{m}}$, so $q(M)$ intersects every basic open subset of $\prod_{\mathfrak{m}\in \mspec(R)}M/\mathfrak{m}M$, and hence is dense in the product space.
\end{proof}

\begin{remark}\label{R12}
	\begin{enumerate}
		\item Assuming the axiom of choice, one can show that the space $\prod_{\mathfrak{m}\in \mspec(R)}M/\mathfrak{m}M$ with the product topology, as defined above, is a $\kappa_M$-maximal Baire space and hence satisfies the equality $\sigma_{\tau}(\prod_{\mathfrak{m}\in \mspec(R)}M/\mathfrak{m}M, R)=\kappa_M$, where $\kappa_M= \min_{\mathfrak{m}\in S_M} |R/\mathfrak{m}|+~1$.\smallskip
		\item  \textbf{$M$ is a $\kappa$-Baire space $\iff$ $M/\jac(M)$ (or $q(M)$) is a $\kappa$-Baire space:}
		
		Every open subset of $M/\jac(M)$ is dense in $M/\jac(M)$, since the same is true for $q(M)$. Thus, the space $M/\jac(M)$ being a $\kappa$-Baire space, is equivalent to the non-emptiness of intersection of any family of less than $\kappa$-many open subsets. By this observation and the definition of topology on $M$, one sees that the surjective map $q': M\rightarrow M/\jac(M)$ is an open map, and hence that $M$ is a $\kappa$-Baire space if and only if $M/\jac(M)$ is a $\kappa$-Baire space.\smallskip
		\item By the definition of the topology on $M$, all proper submodules of $M$ are contained in proper closed submodules. This is because every maximal submodule $K$ of $M$ is closed under the induced Zariski topology. To see this, note that $q_{\mathfrak{m}}(K)=(K+\mathfrak{m}M)/\mathfrak{m}M$ is a subspace of $M/\mathfrak{m}M$ and is hence a closed subset of $M/\mathfrak{m}M$ with the induced Zariski topology. Thus, if $q: M\rightarrow \prod_{\mathfrak{m}\in \mspec(R)}M/\mathfrak{m}M$ is the natural map, inducing the topology on $M$, then $K=K+\jac(M)=\bigcap_{\mathfrak{m}\in \mspec(R)}q_{\mathfrak{m}}^{-1}((K+\mathfrak{m}M)/\mathfrak{m}M)=q^{-1}(\prod_{\mathfrak{m}\in \mspec(R)}q_{\mathfrak{m}}(K))$ is the inverse image of a closed subset of $\prod_{\mathfrak{m}\in \mspec(R)}M/\mathfrak{m}M$, under a continuous map. Thus, recalling the definition of $\sigma_{\tau}(M, R)$, we see that if $M$ is a $\kappa$-maximal Baire space, then $\sigma(M, R)\geq \sigma_{\tau}(M, R)=\kappa$. In particular, we see that if $M$ is a $\kappa_M$-maximal Baire space for $\kappa_M=\min_{\mathfrak{m}\in S_M}|R/\mathfrak{m}|+1$, then $\sigma(M, R)=\kappa_M=\sigma_{\tau}(M, R)$.
	\end{enumerate}
\end{remark}

 Theorem $1.24(ii)$ of \cite{R1977} shows that a dense subspace $X$ of an $\aleph_1$-Baire space $Y$ is also an $\aleph_1$-Baire space, if and only if every $G_\delta$ subset of $Y$ contained in $Y\setminus X$ is nowhere dense in $Y$. One can see that the same proof can be replicated (as commented on pg.$64$ of \cite{R1977}) to prove a $\kappa$-analogue for any infinite cardinal $\kappa$, as follows:

\begin{lemma}\label{L7}
	Let $\kappa$ be an infinite cardinal. A dense subspace $X$ of a $\kappa$-Baire space $Y$ is also a $\kappa$-Baire space, if and only if every $G_{\delta, \kappa}$ subset of $Y$ contained in $Y\setminus X$ is nowhere dense in $Y$. Here, a $G_{\delta, \kappa}$ set is simply the intersection of a family $\mathcal{F}$ of open sets of $Y$, such that $|\mathcal{F}|<\kappa$.
\end{lemma}

By the definition of the induced Zariski topology on $\prod_{\mathfrak{m}\in \mspec(R)}M/\mathfrak{m}M$, it follows that every non-empty open subset is dense in the space. Since $\prod_{\mathfrak{m}\in \mspec(R)}M/\mathfrak{m}M$ is a $\kappa_M$-Baire space under this topology, it follows that every $G_{\delta, \kappa_M}$ subset of $\prod_{\mathfrak{m}\in \mspec(R)}M/\mathfrak{m}M$ is dense in the space. Thus, $q(M)$ is a $\kappa_M$-Baire space if and only if the complement of $q(M)$ in $\prod_{\mathfrak{m}\in \mspec(R)}M/\mathfrak{m}M$ does not contain any $G_{\delta, \kappa_M}$ subset. The open subsets of the induced Zariski topology on $\prod_{\mathfrak{m}\in \mspec(R)}M/\mathfrak{m}M$ being dense in the topology is equivalent to the fact that closed subsets of the space $\prod_{\mathfrak{m}\in \mspec(R)}M/\mathfrak{m}M$ are nowhere dense. Now the complement of $q(M)$ containing a $G_{\delta, \kappa_M}$ subset, is equivalent to $q(M)$ being contained in the union of less than $\kappa_M$-many closed subsets of $\prod_{\mathfrak{m}\in \mspec(R)}M/\mathfrak{m}M$. 

\begin{definition}(See \cite[pp. 64]{R1977}.)\label{D7}
	For any infinite cardinal $\kappa$, a space $X$ is said to be of \textit{first $\kappa$-category} if it can be written as a union of less than $\kappa$-many nowhere dense subsets.
\end{definition}

In light of this extension of the notion of first category spaces, and the fact that subsets of nowhere dense sets are nowhere dense,  it follows from Proposition~\ref{P16} and Lemma~\ref{L7} that:

\begin{corollary}\label{C14}
	The subspace $q(M)$ of $\prod_{\mathfrak{m}\in \mspec(R)}M/\mathfrak{m}M$ is a $\kappa_M$-Baire space if and only if it is not of first $\kappa_M$-category.
\end{corollary}

Since $q(M)$ is a dense subspace of the $\kappa_M$-maximal Baire space $\prod_{\mathfrak{m}\in \mspec(R)}M/\mathfrak{m}M$ for $\kappa_M=\min_{\mathfrak{m}\in S_M}|R/\mathfrak{m}|+1$, it follows that if $q(M)$ is a $\aleph$-Baire space, then $\aleph\leq \kappa_M$. In particular, if $q(M)$ is a $\kappa_M$-Baire space, then it is also a $\kappa_M$-maximal Baire space. By Remark~\ref{R12}($3$), we then have:

\begin{corollary}\label{C15}
	Let $M$ be a finitely generated $R$-module, with the induced Zariski topology, and let $\kappa_M=\min_{\mathfrak{m}\in S_M}|R/\mathfrak{m}|+1$. If the image of $M$ under the map $q: M\rightarrow\prod_{\mathfrak{m}\in \mspec(R)}M/\mathfrak{m}M$ is not of first $\kappa_M$-category, then $\sigma_{\tau}(M, R)=\kappa_M$. In particular, $\sigma(M, R)= \kappa_M$.
\end{corollary}

\noindent We now study the special case when $q(M)=\prod_{\mathfrak{m}\in \mspec(R)}M/\mathfrak{m}M$, and make some observations.

\begin{enumerate}
	
	\item \textbf{Compactness of $M\iff$ Compactness of $M/\jac(M)$ (or $q(M)$):}
	
	The map $q: M\rightarrow \prod_{\mathfrak{m}\in \mspec(R)}M/\mathfrak{m}M$ can be factored into the projection map $q': M\twoheadrightarrow M/\jac(M)$, and the isomorphism $h: M/\jac(M)\rightarrow q(M)\subset \prod_{\mathfrak{m}\in \mspec(R)}M/\mathfrak{m}M$. Since $h$ is a bijection, it is a homeomorphism under the topology. Next, we topologize $M$ by declaring the open sets to be exactly the inverse images under $q'$ of the open sets of $M/\jac(M)$. Clearly, $M$ compact implies $M/\jac(M)$ compact, since it is the continuous image of a compact space. If $M/\jac(M)$ is compact, and $\lbrace U_i\rbrace_{i\in I}$ is any open cover of $M$, then $\lbrace q(U_i)\rbrace_{i\in I}$ is an open cover of $M/\jac(M)$ since $q'$ is an open map, and hence the inverse image of a finite subcover of $\lbrace q(U_i)\rbrace_{i\in I}$ provides a finite subcover of $\lbrace U_i\rbrace_{i\in I}$ 
	. Hence, the compactness of $M/\jac(M)\cong q(M)$ implies the compactness of $M$.\smallskip
	
	\item \textbf{Surjectivity of the map $q$}: Under the assumption that $S_M$ equals $\mspec(R)$, the converse holds as well.
	
	Assume $M$ is compact. Then for any element $(x_\mathfrak{m})_{\mathfrak{m}\in \mspec(R)}$ in the product space  $\prod_{\mathfrak{m}\in \mspec(R)}M/\mathfrak{m}M$, consider the collection $\lbrace q_{\mathfrak{m}}^{-1}(x_{\mathfrak{m}})\rbrace_{\mathfrak{m}\in \mspec(R)}$ of closed subsets, where $q_{\mathfrak{m}}: M\rightarrow M/\mathfrak{m}M$ are the natural continuous projection maps. For any finite subset $\sigma\subset \mspec(R)$, the map $q_\sigma:M\rightarrow \prod_{\mathfrak{m}\in\sigma}M/\mathfrak{m}M$ is surjective by the Chinese Remainder theorem, so the intersection of any finite sub-collection of $\lbrace q_{\mathfrak{m}}^{-1}(x_{\mathfrak{m}})\rbrace_{\mathfrak{m}\in \mspec(R)}$ is non-empty. Since $M$ is compact, by the finite intersection property the intersection of all sets in $\lbrace q_{\mathfrak{m}}^{-1}(x_{\mathfrak{m}})\rbrace_{\mathfrak{m}\in \mspec(R)}$ is non-empty. Thus, if $M$ is a compact space with the induced topology, and $S_M=\mspec(R)$, then the map $q: M\rightarrow \prod_{\mathfrak{m}\in \mspec(R)}M/\mathfrak{m}M$ is surjective.
\end{enumerate}

Summarizing these observations, one has the following.

\begin{proposition}\label{P17}
	If $q(M)=\prod_{\mathfrak{m}\in \mspec(R)}M/\mathfrak{m}M$, i.e., $q$ is surjective, then $M$ is compact. Under the assumption that the set $S_M$ equals $\mspec(R)$, the converse holds: if $M$ is a compact space under the induced Zariski topology, then $q(M)=\prod_{\mathfrak{m}\in \mspec(R)}M/\mathfrak{m}M$, i.e., $q$ is surjective. 
\end{proposition}

Thus, assuming $S_M=\mspec(R)$, there exists an $R$-linear homeomorphism between $M/\jac(M)$ and the inverse limit $\varprojlim_{\sigma\in\mathcal{F}}M_{\sigma}$ of the system $((M_{\sigma})_{\sigma\in\mathcal{F}}, (p_{ij})_{\sigma_i\preccurlyeq\sigma_j\in\mathcal{F}})$ if and only if $M$ is compact. 
One gets the following corollary immediately from the above analysis:

\begin{corollary}\label{C16}
	Assume that the set $S_M$ equals $\mspec(R)$. If $M$ is a compact space under the induced Zariski topology, then we have the equality $\sigma_{\tau}(M, R)=\sigma(M, R)=\min_{\mathfrak{m}\in \mspec(R)}|R/\mathfrak{m}|+1$.
\end{corollary}

When $S_M=\mspec(R)$, the analysis evidently involves a better behaved topology. For any finitely generated module $M$ over a ring $R$, if $S_M\neq \mspec(R)$, then one can consider the multiplicative subset $T=R\setminus\bigcup_{\mathfrak{m}\in S_M}\mathfrak{m}$, and localize to get a finitely generated module $T^{-1}M$ over the ring $T^{-1}R$. Using that every maximal ideal of $T^{-1}R$ is of the form $T^{-1}\mathfrak{m}$ for $\mathfrak{m}\in \mspec(R)$ such that $\mathfrak{m}\cap T=\emptyset$, the maximal ideals of $T^{-1}M$ are precisely $T^{-1}\mathfrak{m}$ for $\mathfrak{m}\in S_M\subsetneq \mspec(R)$. 

Let $S_{T^{-1}M}:=\lbrace T^{-1}\mathfrak{m}\in \mspec(T^{-1}R):$ $\dim_{(T^{-1}R/T^{-1}\mathfrak{m})}(T^{-1}M/(T^{-1}\mathfrak{m})(T^{-1}M))\geq2\rbrace$. We claim $S_{T^{-1}M}=\mspec(T^{-1}R)$. To see this, note that for all $\mathfrak{m}\in S_M$, we have $T^{-1}(\mathfrak{m}M)= (T^{-1}\mathfrak{m})(T^{-1}M)$ and hence $T^{-1}M/(T^{-1}\mathfrak{m})(T^{-1}M)=T^{-1}M/T^{-1}(\mathfrak{m}M)\cong T^{-1}(M/\mathfrak{m}M)$. One can check that the localization map $R/\mathfrak{m}\rightarrow T^{-1}(R/\mathfrak{m})$ (for $\mathfrak{m}\in S_M$) induces an isomorphism between $R/\mathfrak{m}$ and $T^{-1}(R/\mathfrak{m})$ as fields. Hence, the $T^{-1}(R/\mathfrak{m})$-module  $T^{-1}(M/\mathfrak{m}M)$ can be interpreted as an $R/\mathfrak{m}$-module as well. Analogously, the localization map $M/\mathfrak{m}M\rightarrow T^{-1}(M/\mathfrak{m}M)$ induces an isomorphism between the two as $R/\mathfrak{m}$-modules, whence: $$\dim_{(T^{-1}R/T^{-1}\mathfrak{m})}(T^{-1}M/T^{-1}(\mathfrak{m}M))=\dim_{R/\mathfrak{m}}(T^{-1}M/T^{-1}(\mathfrak{m}M))=\dim_{R/\mathfrak{m}}(M/\mathfrak{m}M)$$

Thus, $\dim_{(T^{-1}R/T^{-1}\mathfrak{m})}(T^{-1}M/T^{-1}(\mathfrak{m}M))\geq2$ for all $\mathfrak{m}\in S_M$, i.e. for all $T^{-1}\mathfrak{m}\in \mspec(T^{-1}R)$, leading to $S_{T^{-1}M}=\mspec(T^{-1}R)$. This shows that localization can help remove the ``extra" maximal ideals. Like any module $M$, we can topologize $T^{-1}M$, using the map $q_T: T^{-1}M\rightarrow \prod_{T^{-1}\mathfrak{m}\in\mspec(T^{-1}R)}T^{-1}M/T^{-1}(\mathfrak{m}M)$, where each factor space in the product is equipped with the induced Zariski topology (since it has dimension at least $2$), as an affine space over the field $R/\mathfrak{m}\cong T^{-1}R/T^{-1}\mathfrak{m}$. In fact the topologies defined on $M$ and $T^{-1}M$ are such that the localization map $f:M\rightarrow T^{-1}M$ is a continuous ($R$-linear) map. To see this, consider the diagram:
$$\begin{tikzcd}
M \arrow[d, "f"'] \arrow[rr, "q"] &  &                                  \prod_{\mathfrak{m}\in\mspec(R)}M/\mathfrak{m}M  \arrow[d, "f_{\pi}"]                                              &  & \\
T^{-1}M \arrow[rr, "q_T"']        &  & \prod_{T^{-1}\mathfrak{m}\in\mspec(T^{-1}R)}T^{-1}M/T^{-1}(\mathfrak{m}M) &  &   
\end{tikzcd}$$

\noindent where $f:M\rightarrow T^{-1}M$ is the localization map, $q=(q_\mathfrak{m})_{\mathfrak{m}\in\mspec(R)}$ with $q_\mathfrak{m}: M\rightarrow M/\mathfrak{m}M$ being the natural quotient map and similarly $q_T=(q_{T, \mathfrak{m}})_{T^{-1}\mathfrak{m}\in \mspec(T^{-1}R)}$, with $q_{T,\mathfrak{m}}: T^{-1}M\rightarrow T^{-1}M/T^{-1}(\mathfrak{m}M)$ being the corresponding natural quotient map. The map $f_\pi$ is defined by its component maps, i.e. $f_\pi=(f_{\pi, \mathfrak{m}})_{\mathfrak{m}\in S_M}$, where $f_{\pi, \mathfrak{m}}=\psi_{\mathfrak{m}}\circ p_\mathfrak{m}$ with $p_\mathfrak{m}: \prod_{\mathfrak{m}\in \mspec(R)}M/\mathfrak{m}M\rightarrow M/\mathfrak{m}M$ being the projection map onto a component and $\psi_\mathfrak{m}: M/\mathfrak{m}M\rightarrow T^{-1}M/T^{-1}(\mathfrak{m}M)$ denoting the $R/\mathfrak{m}$-linear isomorphism induced by the localization map, as discussed above. 

One can check the commutativity of the above diagram and hence obtain $q_T\circ f= f_\pi\circ q$. Clearly, the maps $q$ and $q_T$ are continuous. To check the continuity of $f_\pi$, it suffices to check the continuity of each of its component maps, i.e. the maps $\psi_\mathfrak{m}\circ p_\mathfrak{m}$. Clearly $p_{\mathfrak{m}}$ is continuous, and so is the map $\psi_\mathfrak{m}$ (for all $\mathfrak{m}\in S_M$), since every linear isomorphism on vector spaces with the induced Zariski topology is a homeomorphism of topological spaces. Thus, the maps $\psi_\mathfrak{m}\circ p_\mathfrak{m}$ are continuous for all $\mathfrak{m}\in \mspec(R)$. Hence, $f_\pi$ is continuous. By definition, it follows that every open set of $T^{-1}M$ is of the form $q_T^{-1}(\mathcal{U})$ for some open subset $\mathcal{U}\subseteq \prod_{T^{-1}\mathfrak{m}\in\mspec(T^{-1}R)}T^{-1}M/T^{-1}(\mathfrak{m}M)$. Then it follows from the commutativity of the above diagram that $f^{-1}(q_T^{-1}(\mathcal{U}))= q^{-1}(f_\pi^{-1}(\mathcal{U}))$ is an open subset of $M$ due to the continuity of $q$ and $f_\pi$. Hence, $f$ is continuous.

We will show that if $M$ is a $\kappa$-Baire space for any infinite cardinal $\kappa$, then so is the localized module $T^{-1}M$, under its corresponding induced Zariski topology. This uses the following lemma.

\begin{lemma}\label{bairelemma}
The image of a $\kappa$-Baire space under a continuous open map is a $\kappa$-Baire space.
\end{lemma}

We omit the proof of the lemma, as it is the $\kappa$-analogue of Corollary~$4.2$ of \cite{R1977} for ordinary Baire spaces, and can be proved using the $\kappa$-analogue of Theorem~$4.1$ of \cite{R1977}.

\begin{proposition}\label{P18}
    Let $M$ be a finitely generated $R$-module, topologized with the induced Zariski topology. Let $S_M=\lbrace \mathfrak{m}\in \mspec(R):$ $\dim_{R/\mathfrak{m}}(M/\mathfrak{m}M)\geq2\rbrace$ and $T=R\setminus\bigcup_{\mathfrak{m}\in S_M}\mathfrak{m}$. If $M$ is a $\kappa$-Baire space with its induced Zariski topology (for any infinite cardinal $\kappa$), then the $T^{-1}R$-module $T^{-1}M$, with its induced Zariski topology is also a $\kappa$-Baire space.
\end{proposition}

\begin{proof}
Using the above notation, for any basic open set $\prod_{\mathfrak{m}\in \mspec(R)}\mathcal{O}_\mathfrak{m}$ of $\ \prod_{\mathfrak{m}\in \mspec(R)}M/\mathfrak{m}M$, one can see that $f_\pi(\prod_{\mathfrak{m}\in \mspec(R)}\mathcal{O}_\mathfrak{m})= \prod_{\mathfrak{m}\in S_M}\psi_\mathfrak{m}(\mathcal{O}_\mathfrak{m})$, where the indexing of the product by $\mathfrak{m}\in S_M$ and by $T^{-1}\mathfrak{m}\in\mspec(T^{-1}M)$ are equivalent. Since each $\psi_\mathfrak{m}$ is a homeomorphism, $f_\pi(\prod_{\mathfrak{m}\in \mspec(R)}\mathcal{O}_\mathfrak{m})$ is an open subspace of $\prod_{T^{-1}\mathfrak{m}\in\mspec(T^{-1}R)}T^{-1}M/T^{-1}(\mathfrak{m}M)$. Thus $f_\pi$ is an open continuous map. By Remark~\ref{R12}$(2)$, we know that if $M$ is a $\kappa$-Baire space, then so is $q(M)$. Since $f_\pi$ is an open continuous map, it follows using Lemma~\ref{bairelemma} that if $M$ is a $\kappa$-Baire space then so is $f_\pi(q(M))$. Also note that $f_\pi$ is surjective, and since $q(M)$ is dense in $\prod_{\mathfrak{m}\in\mspec(R)}M/\mathfrak{m}M$, it follows that $f_\pi(q(M))$ is dense in $\prod_{T^{-1}\mathfrak{m}\in\mspec(T^{-1}R)}T^{-1}M/T^{-1}(\mathfrak{m}M)$. The relation $f_\pi\circ q= q_T\circ f$ implies that $f_\pi(q(M))$ is a subspace of $q_T(T^{-1}M)$. Thus, if $\mathcal{U}$ is any open subset of $\prod_{T^{-1}\mathfrak{m}\in\mspec(T^{-1}R)}T^{-1}M/T^{-1}(\mathfrak{m}M)$, then $\mathcal{U}\cap q_T(T^{-1}M)$ and $\mathcal{U}\cap f_\pi(q(M))$ are both non-empty, and every open subset of both $q_T(T^{-1}M)$ and $f_\pi(q(M))$ are of such form. 

Since every non-empty open subset of  $\prod_{T^{-1}\mathfrak{m}\in\mspec(T^{-1}R)}T^{-1}M/T^{-1}(\mathfrak{m}M)$ is dense, the same property holds for $q_T(T^{-1}M)$ and $f_\pi(q(M))$. Let $\lbrace \mathcal{O}_i\rbrace_{i\in I}$ be a family of non-empty open subsets of $q_T(T^{-1}M)$. Then there exist open subsets $\mathcal{U}_i$ of  $\prod_{T^{-1}\mathfrak{m}\in\mspec(T^{-1}R)}T^{-1}M/T^{-1}(\mathfrak{m}M)$, for all $i\in I$, such that $\mathcal{O}_i=\mathcal{U}_i\cap q_T(T^{-1}M)$. As $f_\pi(q(M))$ is dense, $\tilde{O}_i=\mathcal{U}_i\cap f_\pi(q(M))$ is a non-empty open subset of $f_\pi(q(M))$ for all $i\in I$. Since $f_\pi(q(M))$ is a $\kappa$-Baire space, if $|I|<\kappa$ then $\bigcap_{i\in I}\tilde{O}_i$ is a dense subspace of $f_\pi(q(M))$, i.e. for any open subset $\mathcal{U}\subseteq\prod_{T^{-1}\mathfrak{m}\in\mspec(T^{-1}R)}T^{-1}M/T^{-1}(\mathfrak{m}M)$, $\bigcap_{i\in I}\tilde{O}_i\cap \mathcal{U}\cap f_{\pi}(q(M))\neq \emptyset$. Clearly $\bigcap_{i\in I}\tilde{O}_i\cap \mathcal{U}\cap f_{\pi}(q(M))\subseteq \bigcap_{i\in I}O_i\cap \mathcal{U}\cap q_T(T^{-1}M)$, thereby proving that $\bigcap_{i\in I}O_i\cap \mathcal{U}\cap q_T(T^{-1}M)\neq\emptyset$, and hence $\bigcap_{i\in I}O_i$ is dense in $q_T(T^{-1}M)$. This proves that $q_T(T^{-1}M)$ is a $\kappa$-Baire space. Hence, so is $T^{-1}M$, by Remark~\ref{R12}$(2)$.
\end{proof}

We next provide a corollary of Proposition~\ref{P18}, which demonstrates its relevance to our problem of interest. We omit the proof, as it is immediate.

\begin{corollary}\label{C17}
Let $M$ be a finitely generated $R$-module (with the induced Zariski topology), which satisfies the equality $\sigma_\tau(M, R)=\sigma(M, R)=\min_{\mathfrak{m}\in S_M}|R/\mathfrak{m}|+1$, where $S_M=\lbrace\mathfrak{m}\in\mspec(R):$ $\dim_{R/\mathfrak{m}}(M/\mathfrak{m}M)\geq 2\rbrace$. Letting $T=R\setminus\bigcup_{\mathfrak{m}\in S_M}\mathfrak{m}$, we have that the $T^{-1}R$-module $T^{-1}M$ (with the induced Zariski topology) also satisfies the equality $$\sigma_\tau(T^{-1}M, T^{-1}R)=\sigma(T^{-1}M, T^{-1}R)=\min_{T^{-1}\mathfrak{m}\in \mspec(T^{-1}R)}|T^{-1}R/T^{-1}\mathfrak{m}|+1 = \min_{\mathfrak{m}\in S_M}|R/\mathfrak{m}|+1.$$
\end{corollary}

For completeness, we end by characterizing when a finitely generated $R$-module $M$ can have a finite Zariski covering number. Here we only assume that the base ring $R$ is commutative and unital, and \textbf{do not need} any assumption on the cardinalities of the residue fields of $R$. Note that dropping this cardinality assumption does not affect the way we topologize $M$. Clearly then, if the covering number is finite then so is the Zariski covering number. The converse is also true:

\begin{proposition}\label{P19}
For a finitely generated $R$-module $M$, the following are equivalent:
\begin{enumerate}
    \item $\sigma_{\tau}(M, R)$ is finite.
    \item $\sigma_{\tau}(M, R) = 2$. (Equivalently, $M$ is reducible in the induced Zariski topology)
    \item $\sigma(M, R)$ is finite.
\end{enumerate}
\end{proposition}

Combined with the inequality $\sigma_{\tau}(M, R)\leq \sigma(M, R)$, Proposition~\ref{P19} has the following consequence: A finitely generated $R$-module $M$ has a countable covering number $\sigma(M,R) = \aleph_0$, if and only if its Zariski covering number is the same: $\sigma_\tau(M,R) = \aleph_0$.

\begin{proof}
The equivalence of $(1)$ and $(2)$ is immediate. Now consider the natural map $q:M\rightarrow\prod_{\mathfrak{m}\in\mspec(R)}M/\mathfrak{m}M$. The equality $\sigma_{\tau}(M, R)=2$ is true if and only if there exist two basic non-empty open subsets $q^{-1}(\prod_{\mathfrak{m}\in\mspec(R)}\mathcal{O}_{\mathfrak{m}})$ and $q^{-1}(\prod_{\mathfrak{m}\in\mspec(R)}\mathcal{U}_{\mathfrak{m}})$ of $M$, such that their intersection is empty, i.e., $\mathcal{W}:=q^{-1}(\prod_{\mathfrak{m}\in\mspec(R)}(\mathcal{O}_{\mathfrak{m}}\cap \mathcal{U}_{\mathfrak{m}}))=\emptyset$. But $\mathcal{W}$ is the inverse image of a basic open set of $\prod_{\mathfrak{m}\in\mspec(R)}M/\mathfrak{m}M$, and $q(M)$ is a dense subset of this product space by Proposition~\ref{P16}. Hence, $\mathcal{W}=\emptyset$ if and only if $\prod_{\mathfrak{m}\in\mspec(R)}(\mathcal{O}_{\mathfrak{m}}\cap \mathcal{U}_{\mathfrak{m}})=\emptyset$, which in turn is possible if and only if $\mathcal{O}_{\mathfrak{m}_0}\cap \mathcal{U}_{\mathfrak{m}_0}=\emptyset$ for some $\mathfrak{m}_0\in S_M$ (assuming axiom of choice). Thus $\mathcal{W}$ is empty if and only if $M/\mathfrak{m}_0M$ is a reducible topological space under its induced Zariski topology for some $\mathfrak{m}_0\in S_M$, which is true if and only if $R/\mathfrak{m}_0$ is finite. From \cite{khare2020}, we know that $R/\mathfrak{m}_0$ is finite for some $\mathfrak{m}_0\in S_M$ if and only if $\sigma(M, R)$ is finite. Thus $(2)$ and $(3)$ are equivalent.
\end{proof}

\subsection*{Acknowledgments}
The author is thankful to Prof.\ Amartya Kumar Dutta, Indian Statistical Institute, Kolkata for introducing him to the toy problem for vector spaces, and for motivating him to explore this in a broader setting and in greater depth. Further, the author is extremely grateful to Prof.\ Apoorva Khare, Indian Institute of Science, Bangalore, for several stimulating and motivating interactions, as well as for a careful and detailed reading of previous versions of this manuscript, which helped greatly improve the exposition.



\end{document}